\documentclass[11pt]{article}
\usepackage{latexsym} 
\usepackage{enumitem}
\usepackage{xcolor}
\usepackage{float}

\usepackage{amsfonts} 
\usepackage{amsmath,amssymb,amsthm,color}
\usepackage{ulem}
\date{} 
\newfont{\fra}{eufm10 scaled 1095}  
\newfont{\Bbg}{msbm10 scaled 1280} 
 
\newcommand\RR{{\mathbb{R}}} 
\newcommand\NN{{\mathbb{ N}}} 
\newcommand\ZZ{{\mathbb{ Z}}}

\newcommand\fg{{\mathfrak{g}}} 
\newcommand\fh{{\mathfrak{h}}} 
\newcommand\fm{{\mathfrak{m}}} 
\newcommand\fn{{\mathfrak{n}}} 
\newcommand\fa{{\mathfrak{a}}} 
\newcommand\fd{{\mathfrak{d}}} 
\newcommand\fp{{\mathfrak{p}}} 
\newcommand\fr{{\mathfrak{r}}}
\newcommand\fs{{\mathfrak s}} 
\newcommand\fz{{\mathfrak{z}}} 

\newcommand\III{{\rm III}} 

\newcommand\eps{\varepsilon} 

\newcommand{\fgl}{\mathop{{\mathfrak{g} \mathfrak {l}}}} 
\newcommand{\fsl}{\mathop{{\mathfrak{sl}}}}
 
\newcommand\fco{\mathfrak c \mathfrak o} 
 
\newcommand{\fso}{\mathop{{\mathfrak s \mathfrak o}}}

\newcommand{\fhI}{\fh^{\rm I}}
\newcommand{\fhIII}{\fh^{\rm III}}
\newcommand\g{{{\mathfrak{g}}^*_2}} 

\newcommand{\GL}{\mathop{{\rm GL}}} 
 
\newcommand{\SO}{\mathop{{\rm SO}}} 
\newcommand{\grO}{{\rm O}}

\newcommand{\G}{{\rm G}_2^*}

\newcommand{\ad}{\mathop{{\rm ad}}} 
\newcommand{\tr}{\mathop{{\rm tr}}} 
\newcommand{\Ad}{\mathop{{\rm Ad}}}

\newcommand{\diag}{\mathop{{\rm diag}}} 
 
\newcommand{\Span}{\mathop{{\rm span}}} 
 
\newcommand{\pro}{{\rm pr}} 

\newcommand\ip{{\langle\cdot \,,\cdot \rangle}} 
 
\newcommand\lb{{[\cdot\,,\cdot]}}

\newcommand{\benur}{\begin{enumerate}[label=(\roman*)]}
\newtheorem{theo}{Theorem}[section] 
\newtheorem{pr}[theo]{Proposition} 
\newtheorem{de}[theo]{Definition} 
 
\newtheorem{re}[theo]{Remark} 
\newtheorem{co}[theo]{Corollary} 
\newtheorem{lm}[theo]{Lemma} 
\newcommand{\MSC}[1]{\par\noindent\textbf{Mathematics Subject Classification (2020):} #1}

\begin{document} 
\title{Left-invariant $\G$-structures of type III}
\author{Viviana del Barco, Ana Cristina Ferreira, Ines Kath}


\maketitle 
\begin{abstract}
\noindent 
We investigate left-invariant $\G$-structures on 7-dimensional Lie groups, focusing on those whose holonomy algebras are indecomposable and of type III, the latter meaning that the socle of the holonomy representation is maximal. Building on the classification of indecomposable holonomy algebras contained in $\g$, we determine which ones arise as infinitesimal holonomy algebras of type III for left-invariant $\G$-structures. Our main result shows that only abelian subalgebras occur, and these are necessarily of dimension two or three. Moreover, we provide explicit Lie groups with left-invariant $\G$-structures realizing these abelian holonomies. 
\end{abstract}

\MSC{Primary 53C29; Secondary 53C50, 53C10, 53C30.}

\section{Introduction}

This article will deal with holonomy groups contained in the non-compact group of type ${\rm G}_2$, which is denoted by $\G$. Such holonomy groups appear on pseudo-Riemannian manifolds of signature $(4,3)$. In particular, the full group $\G$ is a holonomy group, as it can indeed be realized as a holonomy of compact nilmanifolds \cite{FL}. Besides $\G$ itself, also proper subgroups occur as holonomy groups. We will be mainly interested in these subgroups, where we want to consider only connected ones which, moreover, are `genuine' in the sense that the holonomy representation is indecomposable. Instead of studying the holonomy group itself, we consider its Lie algebra, which we call holonomy algebra.  For a Lie subalgebra of $\fso(p,q)$, a necessary condition for it to be a holonomy is that it is a Berger algebra, i.e. that it satisfies Berger's first criterion. In \cite{FK}, the indecomposable  Berger algebras properly contained in the Lie algebra $\g$ of $\G$ were classified. They were distinguished by the dimension of the socle of their natural representation on the tangent space, this dimension is 1, 2 or~3  (see Section~\ref{sec:holt3} for more information). Accordingly, we have three families of  indecomposable Berger algebras, namely those of type I, II or III. For each type, \cite{FK} provides a complete list of all  Berger algebras of that type. All these Berger algebras can indeed be realized as holonomy algebras of local metrics \cite{FKsigma,VolkII,VolkIII}. For some of these holonomy algebras a global metric has also been found. 
Here we mean by `global' that there is a `nice' manifold with such a metric and not just a coordinate patch.   In \cite{Ksymm}, indecomposable indefinite symmetric spaces with $\G$-structure are classified. Their holonomy  algebras are abelian and of dimension two or three. 
In \cite{FL}, Fino and Lujan studied torsion-free  $\G$-structures on nilpotent Lie groups. In addition to the already mentioned example whose holonomy is equal to $\g$, they found an example of a torsion-free  $\G$-structure with $6$-dimensional indecomposable holonomy.  In
\cite{FK}, new left-invariant metrics with indecomposable  holonomy in $\g$ were found, where, for type III, only the 3-dimensional abelian Berger algebra was realized. 
In addition to these examples, it is known that torsion-free left-invariant $\G$-structures whose associated metric is bi-invariant on Lie groups with a cocompact lattice have a trivial holonomy \cite{G}. 

However, the question of which holonomy algebras contained in $\g$ are realizable by a left-invariant metric has not been studied systematically. 
Here we will attack this question for holonomy algebras of type III, i.e. for those whose socle is maximal. We solve it completely under the additional assumption that not only the metric but also the parallel $\G$-structure is left-invariant. More precisely, we determine which of the holonomy algebras of type III listed in \cite{FK} are the Lie algebra of the holonomy group of a Lie group endowed with a left-invariant  torsion-free $\G$-structure and its associated metric. 

We can state our main result as follows.

\begin{theo} \label{th:main}
    Let $\fh$ be a subalgebra of $\g$ whose natural representation on $\RR^{4,3}$ is indecomposable and of type \III. Then $\fh$ is the infinitesimal holonomy of a left-invariant $\G$-structure if and only if $\fh$ is abelian. If $\fh$ is abelian, then it has dimension two or three. 
\end{theo}

Our proof of the ``if'' direction is constructive. It consists of giving a one-parameter family of Lie groups such that all of them carry a parallel left-invariant $\G$-structure whose holonomy is 3-dimensional abelian for non-zero values of the parameter, and 2-dimensional abelian otherwise. It is remarkable that their left-invariant metrics can also be induced by non left-invariant parallel $\G$-structures, as we show in Section~\ref{S7}.

To make a better description of the main result of the paper, let us recall some properties of $\G$ and introduce some notation. Recall that the complex Lie group ${\rm G}_2^{\mathbb C}$ contains two conjugacy classes of 9-dimensional maximal para\-bolic subgroups. These correspond to two conjugacy classes $P_1$, $P_2$ of real subgroups in $\G$. These conjugacy classes can be characterized in the following way. The group $\G$ acts transitively on the set of isotropic lines in $\RR^{4,3}$. Groups in $P_1$ appear as the stabilizer of an isotropic line under this action. Similarly, $\G$ acts transitively on the set of isotropic 2-planes with vanishing cross product and groups in $P_2$ can be understood as the stabilizer of such a 2-plane (see \cite{FK} for more explanations).
Let us denote by $\fp_1$ and $\fp_2$ the conjugacy classes of Lie algebras corresponding to $P_1$ and $P_2$. 
Lie algebras in $\fp_1$ are isomorphic to $\fgl(2,\RR)\ltimes \fm$, where $\fm$ is a 3-step nilpotent Lie algebra, while Lie algebras belonging to $\fp_2$ are isomorphic to $\fgl(2,\RR)\ltimes {\fn}$ for a 5-dimensional Heisenberg algebra $\fn$. 
Holonomy algebras of type III are subalgebras of Lie algebras in $\fp_1$. More exactly, they are contained in a subalgebra isomorphic to $\fgl(2,\RR)\ltimes[\fm,\fm]$. In Section~\ref{sec:holt3} we will give an explicit embedding of $\g$ into $\mathfrak{so}(4,3)$ realizing $\G$ inside $\mathrm{SO}(4,3)$. Then a representative of $P_1$ is given by the stabilizer of the real line spanned by the first basis vector. Its Lie algebra is equal to the Lie algebra $\fhI\cong\fgl(2,\RR)\ltimes\fm$ defined in \eqref{eq:h1} and the subalgebra $\fgl(2,\RR)\ltimes[\fm,\fm]$ is equal to $\fhIII$ defined in \eqref{eq:h3}. In~\eqref{eq:m}, we define two abelian subalgebras $\fm(1,0,1)\subset\fm(1,0,2)=[\fm,\fm]$ of $\fhIII$.
The abelian subalgebras appearing in Theorem \ref{th:main} are conjugate to these two subalgebras.

The paper is organized as follows. In Section 2 we fix the standard embedding of $\G$ into $\SO(4,3)$ and of its Lie algebra that will be used all along the paper. We introduce notation and recall the list of possible indecomposable holonomy algebras of type III obtained in \cite{FK}; all of them are of the form $\fh=\fa\ltimes \fm(1,0,k) $ with $\fa \subset \fgl(2,\mathbb R)$ and $\fm(1,0,k)$ abelian (see Theorem \ref{TFK}). 

Section 3 specializes in left-invariant $\G$-structures on Lie groups. We introduce geometric tensors for later use and show a restriction for the Levi-Civita connection on left-invariant vector fields, when the holonomy of the pseudo-Riemannian manifold is assumed to be contained $\fhIII$.

Section 4 summarizes technical results that will be used in the sequel; these are general results for pseudo-Riemannian Lie groups endowed with a left-invariant metric of signature $(4,3)$, whose Levi-Civita connection restricts to a particular subalgebra of $\fg_2^*$.

Section 5 shows that if $\fh$ is the holonomy algebra of a Lie group endowed with a left-invariant $\G$-structure, then  its $\fa$  part cannot be a non-trivial subalgebra of the Lie algebra of diagonal matrices. This is a first step towards the proof of Theorem \ref{th:main}. Indeed, Section 6 contains the second and final step of the ``only if'' part, that shows that actually $\fa$ has to be zero dimensional and thus $\fh$ is abelian.  

Finally, Section 7 is devoted to  constructing explicit examples of Lie groups with left-invariant $\G$-structures whose holonomies are the remaining two abelian Lie algebras that are mentioned in our main theorem above. 
The detailed proof of Theorem \ref{th:main} is also part of this section.

\noindent \textbf{Acknowledgments.} We would like to thank the Banff International Research Station (Banff, Canada) for its support during the Women in Geometry 3 Workshop (23w5062), where the work on this paper was first developed.

The research of ACF was supported by Portuguese Funds through FCT (Fundação para a Ciência e a Tecnologia, I.P.) within the Projects UIDB/00013/2020, UIDP/00013/ 2020 and UID/00013: Centro de Matemática da Universidade do Minho (CMAT/UM).
   
VdB is partially supported by FAPESP grant 2023/15089-9.  VdB and ACF would like to thank  the Department of Mathematics of Philipps Universität Marburg and VdB the Centro de Matemática da Universidade do Minho for their hospitality during the visits on which this work was advanced.

IK's stay at the Banff International Research Station was supported by the Programme: Kongressreisen 23, funded by  Deutscher Akademischer Austauschdienst (DAAD).

\section{Holonomy algebras of type {\rm III}}\label{sec:holt3}

The action of $\GL(7,\RR)$ on $\bigwedge^3\RR^7$ has two open orbits. Elements of these orbits are called generic. Their stabilizers are isomorphic to ${\rm G}_2$ for one orbit and to $\G$ for the other one. A generic three-form $\omega$ induces an orientation and a non-degenerate symmetric inner product $\ip$ on $\RR^7$. If the stabilizer of $\omega$ equals $\G$, then this inner product has signature $(4,3)$. In particular, $\G$ is contained in $\SO(\RR^7,\ip)\cong\SO(4,3)$. 

We want to fix a standard embedding of $\G$ into $\SO(4,3)$ and describe its Lie algebra. Let $e_1,\dots, e_7$ be the standard basis of $\RR^7$. The 3-form 
\begin{equation}\label{Eomega}
\omega_0=\sqrt 2(e^{167}+e^{235})-e^4\wedge (e^{15}-e^{26}-e^{37})\,
\end{equation}
is generic and will serve as the standard 3-form here. It induces the metric 
\begin{eqnarray}\label{Eip}
\ip&=&2(e^1\cdot e^5+e^2\cdot e^6+e^3\cdot e^7)- (e^4)^2.
\end{eqnarray}
The stabilizer of $\omega_0$ in $\GL(7,\RR)$ is the standard embedding of $\G$ that we will use in this article. Its Lie algebra $\g\subset\fso(4,3)$ consists of all elements of the form 
\begin{equation}\label{eq:g2}
\begin{pmatrix}
    s_1+s_4 & -s_{10}  & s_9 &  \sqrt{2} s_6 & 0 & -s_{11} & -s_{12} \\
    -s_8 & s_1 & s_2 & \sqrt{2} s_9 & s_{11} & 0 & s_6 \\
    s_7 & s_3 & s_4 & \sqrt{2} s_{10} & s_{12} & -s_6 &  0 \\
\sqrt{2} s_5 & \sqrt{2} s_7 & \sqrt{2} s_8 & 0 & \sqrt{2} s_6 & \sqrt{2} s_9 & \sqrt{2} s_{10}\\
0 & s_{13} & s_{14} & \sqrt{2} s_5 & -s_1-s_4 & s_8 & -s_7\\
-s_{13} & 0 & -s_{5} & \sqrt{2} s_7 & s_{10} & -s_1 & -s_3\\
-s_{14} & s_5 & 0 & \sqrt{2} s_8 & -s_9 & -s_2 & -s_4  
\end{pmatrix}
\end{equation}
for $s_1, \dots, s_{14} \in \RR$.

Let us introduce some general notions before we study subalgebras of $\g$. Let $\fh$ be a subalgebra of $\fso(p,q)$ on the pseudo-Euclidean space $\RR^{p,q}$. Its natural representation on $\RR^{p,q}$ is called {\it indecomposable} if it does not leave invariant any proper non-zero subspace which is non-degenerate with respect to the pseudo-Euclidean inner product. The {\it socle} of a representation is its maximal semisimple subrepresentation. Suppose
that the natural representation of $\fh\subset\fso(p,q)$ is indecomposable but not irreducible. Then its socle is totally isotropic. In this case the dimension of the socle will be called the {\it type} of $\fh$.

Now let us assume that $\fh$ is a subalgebra of $\g$ whose natural representation of $\RR^{4,3}$ is indecomposable but not irreducible. Since the socle is totally isotropic, its dimension is 1, 2, or 3. Accordingly, the type of $\fh$ is denoted by I, II or III.

The following subalgebra $\fhI\subset\fg_2^*$ is an example of a subalgebra of type I: 
\begin{equation}\label{eq:h1} \fhI:=\{ h(A,v,u,y)\mid A\in\fgl(2,\RR),\ v\in\RR,\ u, y \in\RR^2\},\
\end{equation}
where
\begin{equation}
    \label{eq:havuy}
h(A,v, u,y):= \left(
\begin{array}{ccccccc}
\tr (A) &-u_2&u_1&\sqrt 2 v&0&-y_1&-y_2\\
0&a_1&a_2&\sqrt 2 u_1&y_1&0&v\\
0&a_3&a_4&\sqrt 2 u_2&y_2&-v&0\\
0&0&0&0&\sqrt 2 v&\sqrt 2 u_1&\sqrt 2 u_2\\
0&0&0&0&-\tr( A)&0&0\\
0&0&0&0&u_2&-a_1&-a_3\\
0&0&0&0&-u_1&-a_2&-a_4
	\end{array}\right),
\end{equation}
for	$ A=\left(\begin{array}{cc} a_1&a_2\\a_3&a_4\end{array}\right)$, $y=(y_1,y_2)^\top$, $u=(u_1,u_2)^\top$.  In the following $y$ and $u$ will always be understood as vertical vectors, but we will drop the $\top$ sign. The Lie algebra $\fhI$ is isomorphic to $\fgl(2,\RR)\ltimes_\rho \fm$, where $\fm$ is equal to the 3-step nilpotent Lie algebra
\[\fm=\{ h(0,v,u,y)\mid v\in\RR,\ u,y \in\RR^2\} \]
and the representation of $\fgl(2,\RR)$ on $\fm$ is given by
\begin{equation}\label{eq:rhoh1}
\rho(A)(v,u,y)=(\tr (A)v, Au, (A+\tr(A)I_2)y).
\end{equation}
For a subalgebra $\fh\subset\fhI$, we define 
\begin{equation}\label{eq:fa}\fa:=\fa(\fh):=\{A\in\fgl(2,\RR)\mid h(A,v,u,y)\in\fh \mbox{ for some }v\in\RR,\,u,y\in\RR^2\}\end{equation}
and identify $\fa$ with the subalgebra $\{h(A,0,0,0)\mid A\in\fa\}$ of $\fhI$. 

We thus have defined the projections $\pro_A:\fh\to \fa$, $\pro_v:\fh\to \RR$ and $\pro_u,\pro_y:\fh\to \RR^2$, such that for $h=h(A',v',u',y')\in \fhI$, $A'=\pro_A (h)$, $v'=\pro_v (h)$, $u'=\pro_u(h)$ and $y'=\pro_y (h)$.

The Lie algebra $\fhI$ is a representative of the  conjugacy class $\fp_1$ of maximal parabolic subalgebras mentioned in the introduction. Since $\G$ acts transitively on the set of isotropic lines, any subalgebra $\fh\subset\g$ of type I is conjugate to a subalgebra of $\fhI$. 

We further define the following subalgebra of type III:
\begin{equation}\label{eq:h3} \fhIII:=\{ h(A,v,y):=h(A,v,0,y)\mid A\in\fgl(2,\RR),\ v\in\RR,\  y \in\RR^2\}\subset \fhI.
\end{equation} 
We have that $\rho$ preserves the abelian subalgebra $[\fm,\fm]=\{ h(0,v,y)\mid v\in\RR,\ y \in\RR^2\}$ and 
\begin{equation}\label{eq:sdp}
\fhIII=\fgl(2,\RR)\ltimes_\rho [\fm,\fm].
\end{equation} 
If $\fh\subset\g$ is of type III, then 
it is conjugate by an element $g\in \G$ to a subalgebra of $\fhIII$. Indeed, $\fh$ leaves invariant the isotropic three-dimensional socle. But then, by \cite[Lemma 2.2]{FK}, it leaves invariant also an isotropic line contained in the socle and therefore also a complementary two-dimensional subspace of the socle. For more detailed information see~\cite[Lemma 2.10]{FK}.

The converse is not always true. In fact, a subalgebra $\fh\subset\fhIII$ admitting an indecomposable representation on $\RR^{4,3}$ is of type III only if its  subrepresentation on $\Span\{e_2,e_3\}$ is semisimple.

For later use, we adopt the following notation from \cite{FK}. For $k=1,2$, let $\fm(1,0,k)\subset\fhIII$ be the abelian subalgebras 
\begin{equation}\label{eq:m}
\fm(1,0,1)=\{h(0,v,(y_1,0))\mid v,y_1\in\mathbb R\},\ \
\fm(1,0,2)=\{h(0,v,y)\mid v\in \mathbb R,\,y\in\mathbb R^2\}, 
\end{equation}
which are invariant under $\rho$.

\begin{lm}\label{lm:abelian}
    Let $\fh\subset\g$ be a subalgebra whose natural representation on $\RR^{4,3}$ is indecomposable and of type {\rm III}. If $\fh$ is abelian, then $\fh$ is conjugate to $\fm(1,0,1)$ or $\fm(1,0,2)$.
\end{lm}
\begin{proof}
    Since $\fh$ is of type III, we may assume that it is contained in $\fhIII$. Assume, $\fh$ would contain an element $h(A_0,v_0,y_0)$ such that $\tr(A_0)\not=0$. If we conjugate $\fh$ by $\exp\big(h(0,v_0/\tr(A_0),0)\big)$, we obtain a subalgebra $\fh'$ $\subset \fhIII$, which contains $h_0:=h(A_0,0,y_0)$. Since the natural representation of $\fh'$ is also indecomposable, $\fh'$  contains an element of the form $h_1:=h(A,1,y)$. But $\pro_v([h_0,h_1])=\tr(A_0)$, which is impossible since $\fh$ is abelian. Consequently, $\fa=\fa(\fh')$ is an abelian subalgebra of $\fsl(2,\RR)$. In particular, $\dim\fa\le 1$. Assume that $\dim\fa=1$. Then there exists an element $h(A_0,v_0,y_0)$, $A_0\not=0$. Since $\fh$ is of type III, the adjoint action of $\fa$ is semisimple. Since $\fa=\Span\{A_0\}$, we see that $A_0$ is a semisimple matrix. In particular, $A_0$ is invertible because of $\tr(A_0)=0$. Conjugating $\fh$ by $\exp\big(h(0,0,A_0^{-1}y_0)\big)$ we obtain a subalgebra $\fh'\subset\fhIII$, which contains $h_0:=h(A_0,v_0,0)$. For each element $h_1=h(0,v,y)\in\fh'$, we have  $0=\pro_y([h_0,h_1])=A_0(y)$. Since $A_0$ is invertible, this implies $\pro_y(\fh')=0$. Hence the non-degenerate subspace $\Span\{e_1,e_4,e_5\}$ is invariant under $\fh'$, which contradicts the indecomposability of the natural representation of $\fh'$. We obtain $\fa=0$. If $\fh'=\fm(1,0,2)$, we are done. If not, then $\fh'=\Span\{h(0,1,y_0),h(0,0,y)\}$, $y\not=0$, by indecomposability. Conjugating first by $\exp(h)$ for $h=h(0,0,-y_0/3,0)\in\fhI$ and after that by a suitable element $\exp\big(h(A,0,0)\big)$, we obtain $\fm(1,0,1)$.
\end{proof}

Let $M$ be a 7-dimensional manifold. A $\G$-structure on $M$ is a reduction $P_{\G}(M)$ of the frame bundle $P(M)$ to $\G$. Equivalently, a $\G$-structure is a 3-form $\omega$ on $M$ which is generic at each point of $M$ and whose stabilizer is isomorphic to $\G$. The correspondence is given by 
\[P_{\G}(M)=\{ (e_1,\dots,e_7)\in P(M)\mid \omega = \omega_0 \mbox{ w.r.t. } e_1,\dots,e_7\}. \]
The bundle $P_{\G}(M)$ defines a pseudo-Riemannian metric $g$ of signature $(4,3)$ on $M$. This metric is given by \eqref{Eip} with respect to every frame $(e_1,\dots,e_7)\in P_{\G}(M)$. The $\G$-structure is called parallel if the Levi-Civita connection $\nabla^g$ of $g$ reduces to $P_{\G}(M)$ or, equivalently, if $\nabla^g\omega=0$. 

We now assume that we are given a parallel $\G$-structure $P_{\G}(M)$ on $M$ together with the reduced Levi-Civita connection. We fix a point $x\in M$. Let $H_x$ denote the holonomy group of  $P_{\G}(M)$ at $x$. Let $\mathfrak{e}=(e_1,\dots,e_7)$ be a frame at $x$ which belongs to $P_{\G}(M)$.  Using this frame, we can identify $H_x$ with a subgroup 
of $\G$. The conjugacy class of this subgroup 
in $\G$ depends neither on $x$ nor on the chosen basis $\mathfrak{e}$. So each element of this congugacy class is called a holonomy group of the given $\G$-structure and its Lie algebra is called a holonomy algebra.

The natural representation of a holonomy algebra $\fh\subset \g$ 
 on $\RR^{4,3}$ is called holonomy representation. A holonomy algebra is called indecomposable if its holonomy representation is indecomposable.

\begin{de}[Type of the holonomy] An indecomposable holonomy $\fh\subset \fg_2^*$ is of type {\rm I}, {\rm II}, or {\rm III}, if the socle of the holonomy representation has dimension $1$, $2$ or $3$,  respectively.
\end{de}
Of course, the type of a holonomy algebra only depends on its conjugacy class.

In this paper we are interested in holonomy algebras $\fh$ of type III.  Let $\fh$ be such a holonomy algebra. As we have seen above, there exists an element $g\in\G$ such that $\Ad(g)(\fh)\subset \fhIII$. So we may assume that $\fh$ is contained in $\fhIII$.  

An important tool in the study of holonomy groups is Berger's first criterion, which we want to recall here. For a subalgebra $\fh\subset \fso(p,q)$,  
\[
{\mathcal K}(\fh)
=\{
R\in{\textstyle\bigwedge^2}(\RR^{p,q})^*\otimes \fh\mid\forall x,y,z\in \RR^{p,q}: R(x,y)z+R(y,z)x+R(z,x)y=0
\}
\]
is the space of formal curvature tensors with values in $\fh$. If $\fh$ is a holonomy algebra, then 
\[\fh=\Span\{R(x,y)\mid x,y\in\RR^{p,q},\ R\in{\cal K}(\fh)\}.\]

For a holonomy algebra $\fh\subset\fhIII$, we obtain in particular $$\fh\subset\Span\{R(x,y)\mid x,y\in\RR^{4,3},\ R\in{\cal K}(\fhIII)\}.$$
So, let us describe ${\cal K}(\fhIII)$.
For an algebraic curvature tensor $R\in {\cal K}(\fhIII)$ we put 
$R_{ij}:=R(e_i,e_j)$, where as above $e_1,\dots,e_7$ is the standard basis of $\RR^{4,3}$, and denote $A_{ij}:=\pro_A(R_{ij})$.

The following result is an immediate consequence of \cite[Table 1]{FK}, see also \cite[Table 1]{VolkIII}.

\begin{pr}\label{Pholtab1} The space ${\mathcal K}(\fhIII)$ of algebraic curvature tensors can be parametrized by $a_1,a_2,a_3\in\RR$, $b_k,j_k\in\RR$ ($k=1,\dots,4$) and $c_3,c_4,w_1,w_2,t\in\RR$, where $R=h(A,v,y)\in  {\mathcal K}(\fhIII)$ is given by the data in Table 1 in Appendix~\ref{Acurv}.
\end{pr}

Besides the well known Lie algebra $\fsl(2,\RR)$, the Lie algebras
$$\fco(2):=\left\{\left. \begin{pmatrix}
    a&-b\\b&a 
\end{pmatrix}\ \right|\ a\in\RR\right\} ,\quad \fd=\{\diag(d_1,d_2)\mid d_1, d_2\in\RR\}$$ will appear in the theorem below. Furthermore, we will use the abelian Lie algebras $\fm(1,0,k)$, $k=1,2$ introduced in \eqref{eq:m}.

\begin{theo}\label{TFK}{\rm \cite{FK,FKsigma,VolkII,VolkIII}}
A subalgebra $\fh\subset\fg_2^*$  is an indecomposable holonomy algebra of type {\rm III} if and only if it is conjugate by an element of $\G$ to one 
 of the following Lie algebras:
\begin{enumerate}
\item $\fa\ltimes \fm(1,0,2)$ with $\fa\in \{\fsl(2,\RR),\, \fgl(2,\RR),\, \fco(2),\fd\}$,
\item $\fa\ltimes \fm(1,0,k)$ for some $k\in\{1,2\}$ with $\fa\in \{0,\, \RR\cdot\diag(1,0)\}$.\end{enumerate}
\end{theo}

\section{Left-invariant $\G$-structures and their holonomy}

Let $G$ be a 7-dimensional Lie group endowed with a parallel left-invariant $\G$-structure $\omega$. Let $\ip$ be the induced left-invariant metric on $G$ and let $\ip$ denote also the corresponding metric on the Lie algebra $\fg$ of $G$.

For $X\in\fg$, let $\hat X$ be the right-invariant (i.e., the fundamental) vector field on $G$. We use Wang's description of invariant connections by
$\Lambda(X)=(\nabla_{\hat X}-{\cal L}_{\hat X})_e$. Then $\Lambda(X)Y=(\nabla_{\hat X}Y)_e$, where we identify elements of $\fg$ with left-invariant vector fields. Since the commutator of left- and right-invariant vector fields vanishes, we obtain $$\Lambda(X)Y= \nabla_{X}Y, \quad \forall \,X,Y\in \fg.$$ Thus $\Lambda$ can be seen as linear map $\Lambda:\fg\to \fso(\fg,\ip)$.

Let $R$ denote the curvature tensor of $(G,\ip)$, which is also left-invariant and thus can be seen as a linear tensor on $\fg$. Explicitly, 
\[R(X,Y)=[\Lambda(X),\Lambda(Y)]-\Lambda([X,Y]) \]
for $X,Y\in\fg$. Let $\fh_e\subset\fso(\fg,\ip)$ denote the holonomy algebra of the pseudo-Riemannian Lie group $(G,\ip)$ at $e$.  We define
\begin{equation}\label{eq:m0}
    \fr_0:= \Span\{R(X,Y)\mid X,Y\in \fg\}=\Span\{ [\Lambda(X),\Lambda(Y)]-\Lambda([X,Y])\mid  X,Y\in\fg\}.
\end{equation}    
According to \cite[Cor. 4.2.]{KN2}, the holonomy algebra $\fh_e$  is given by
\begin{equation} \fh_e = \fr_0+[\Lambda(\fg),\fr_0]+ [\Lambda(\fg),[\Lambda(\fg),\fr_0]]+\dots. \label{eq:KNhol} \end{equation}

Let $\mathfrak{e}=(e_1, \ldots, e_7)$ be a basis of its Lie algebra $\fg$ such that $\omega=\omega_0$, where $\omega_0$ is as in \eqref{Eomega}, then $\ip$ is as in~\eqref{Eip}.  Using this basis we identify $\fh_e$ with a subalgebra $\fh\subset\g$, which is also called holonomy algebra. If it is necessary to be more precise, we will say that $\fh$ is the holonomy with respect to $\mathfrak{e}$. We use the basis also to identify $\Lambda(X)\in\fso(\fg,\ip)$, $X\in\fg$, with a map in $\fso(4,3)$, which we also denote by $\Lambda(X)$. The resulting map $\Lambda$ with values in $\fso(4,3)$ will also be called Levi-Civita connection with respect to the basis $\mathfrak{e}$. Since our $\G$-structure is parallel and left-invariant, we have $0=\nabla_X\omega=\Lambda(X)^*\omega$ if $X\in \fg$, which implies $\Lambda(\fg)\subset\g$, i.e. $\Lambda(X)\in\g$ for all $X\in\fg$.

The following result gives properties on the Levi-Civita connection of $(G,\ip)$ by assuming that the holonomy algebra is contained in the Lie algebra $\fhIII$ defined in \eqref{eq:sdp}.

\begin{pr}\label{pro:educg}
 Suppose that 
 $\fh$ above is one of the Lie algebras listed in Theorem~\ref{TFK}, and thus $\fh=\fa\ltimes\fm(1,0,k)\subset\fhIII$. Then:
\begin{enumerate}
\item If $\fa$ is one of the Lie algebras $\fsl(2,\RR)$, $\fgl(2,\RR)$, ${\fco}(2)$ or $\fd$, then $\Lambda(\fg)\subset\fhIII$. Moreover, if $\fa\in\{ \fgl(2,\RR),\fco(2),\fd \}$, then $\pro_A(\Lambda(X))\in \fa$ for all $X\in\fg$.

\item If $\fa=\RR\cdot\diag(1,0)$, then $\Lambda(\fg)\subset \{h(A,v,u,y)\mid A\in\fd,\ u=(0,u_2)\}\subset\fh^I$.
\end{enumerate}
\end{pr} 

\begin{proof} Since $G$ is equipped with a parallel left-invariant $G_2^*$-structure, we have that $\Lambda(\fg)\subset \g$.  Moreover, from \eqref{eq:KNhol}, $[\Lambda(X),\fh]\subset\fh$. Therefore, for all $X\in\fg$, $\Lambda(X)$ is in the normalizer of $\fh$ inside $\g$. Let $N \in\g$ be in the normalizer of $\fh$. Then conjugation by $\exp (tN)$ maps $\fh$ to $\fh$. In particular, $\exp(tN)$ maps the socle to the socle. Thus the same is true for $N$. Since for all Lie algebras in Theorem~\ref{TFK}, the socle is spanned by $e_1$, $e_2$ and $e_3$, we obtain $$N\in \big\{h\in\g\mid h(\Span\{e_1,e_2,e_3\})=\Span\{e_1,e_2,e_3\}\big\}=\fhI$$
using~\eqref{eq:g2}. Hence $N=h(A,v,u,y)$, for some $h(A,v,y,y) \in \fhI$, thus
$[N,h(\bar A,\bar v,0,\bar y)] = h([A,\bar A],*,\bar Au,*)\in\fh\subset\fhIII$ for all $h(\bar A,\bar v,0,\bar y)\in\fh$. This shows that $A$ belongs to the normalizer of $\fa$ in $\fgl(2,\RR)$, which can be easily determined.  Moreover we see that $u$ must be annihilated by all elements of $\fa$. This proves the claim.
\end{proof}

We end this section by introducing some notation that will be used throughout the rest of the article. For each $i,j=1, \ldots 7$, we denote $\Lambda_i:=\Lambda(e_i)\in \g$ and 
\begin{equation} \label{eq:R}
R_{ij}:=R(e_i,e_j)=[\Lambda_i,\Lambda_j]-\Lambda([e_i,e_j])\in\g.
\end{equation}

\begin{re}\label{rem:Aspa}
Let $\fh\subset\fhIII$ be a holonomy algebra of type {\rm III} and  let $\fa$ be as in \eqref{eq:fa}. Then $\fa$ is contained in the Lie algebra generated by $A_i:=\pro_A(\Lambda(e_i))$, $i=1, \ldots, 7$. In fact, by \eqref{eq:KNhol},
$\fa$ is linearly spanned by  
\begin{equation}\label{eq:Aij}
A_{ij}:=\pro_A(R_{ij}) =\pro_A([\Lambda_i,\Lambda_j]-\Lambda([e_i,e_j]))=[A_i,A_j]-\sum_{k=1}^7c_{ij}^kA_k
\end{equation}
and their Lie brackets with $A_i$, where $c_{ij}^k$ denote the structure coefficients. 
\end{re}

We will frequently use the second Bianchi identity 
\begin{equation}
B_{ijk}:=-\sum_{cycl} (\nabla_{e_i}R)(e_j,e_k)=\sum_{cycl}( R(\Lambda_i e_j,e_k)+R( e_j,\Lambda_i e_k)-[\Lambda_i,R_{jk}])=0,\label{eq:2B}
\end{equation}
for all $i,j,k=1,\dots,7$. We denote by $B_{ijk}(m,n)$ the $(m,n)$-entry of $B_{ijk}\in\g$.

 \section{Metric Lie algebras with Levi-Civita connection in ${\mathfrak h}^{\rm III}$}

Along this section $\fg$ is a Lie algebra endowed with the metric \eqref{Eip} on a given basis $\mathfrak{e}=(e_1, \ldots, e_7)$. Suppose that its Levi-Civita connection $\Lambda$ with respect to $\mathfrak{e}$ satisfies  $\Lambda(\fg)\subset \fhIII$, so that
\begin{equation}
    \label{eq:lami}
    \Lambda_i=\Lambda(e_i)=h(A_i,v_i,(y_{i1},y_{i2})),\ A_i=\begin{pmatrix} a_{i1} & a_{i2} \\ a_{i3}& a_{i4}\end{pmatrix},\ i=1, \ldots, 7.
\end{equation}
In particular, the left-invariant 3-form $\omega_0$ defined by \eqref{Eomega}  is parallel. 

Using the torsion free property of the Levi-Civita connection and the form of $\ip$ in \eqref{Eip}, we get:

\begin{pr} \label{pr:brackets} If $\Lambda(\fg)\subset \fhIII$ then  the Lie bracket relations between the basis elements $e_1, \ldots, e_7$ are given in Table \ref{A1.h3}.
\end{pr}

\begin{proof} 
Using that $[x,y]=\Lambda(x)y-\Lambda(y)x$, for all $x,y\in\fg$, and \eqref{eq:h3}, one can compute the Lie brackets of basis vectors. 
\end{proof}

Notice that the coefficients  appearing in Table \ref{A1.h3} have to satisfy nontrivial relations due to the Jacobi identity.

Let $R$ denote the curvature tensor and let  $\fh$ denote the holonomy Lie algebra of $(\fg,\ip)$ (or of the associated simply connected Lie group with the induced left-invariant pseudo-Riemannian metric) with respect to the basis $\mathfrak{e}$. If  $\Lambda(\fg)\subset \fhIII$ then, clearly, $R\in \mathcal K(\fhIII)$ but also, for every fixed $x\in \fg$, $\nabla_x R \in \mathcal{K}(\fhIII)$ since $\nabla_x R$ has the same symmetries as $R$. Therefore, $R$ and $\nabla_x R$ have the form in Table \ref{tab:t1} by Proposition \ref{Pholtab1}; in addition, by \eqref{eq:KNhol}, $\fh\subset \fhIII$. We use these properties to give necessary conditions for $\fh$ to be indecomposable.

\begin{lm}\label{lm:Rdec}Suppose $\Lambda(\fg)\subset \fhIII$. 
 If $R_{15}=R_{45}=0$ or, equivalently, $R_{15}=0$ and $w_1=w_2=0$, then $\fh$ is decomposable. 

In particular, if $\Lambda_1=0$ and $\Lambda_4=h(0,v_4,0)$, then $\fh$ is decomposable.
\end{lm} 
\begin{proof}
We use Table \ref{tab:t1}: from $R_{15}=0$, we get $\tr(A_{56})=\tr(A_{57})=0$, and  $R_{45}=0$ implies $\pro_v(R_{56})=\pro_v(R_{57})=0$. This implies that for all $h\in \fr_0$, where $\fr_0$ is in \eqref{eq:KNhol}, one has $\tr(\pro_A(h))=0=\pro_v(h)$. By \eqref{eq:rhoh1} also the elements of $[\fr_0,\Lambda(\fg)]$ have this property. By induction we obtain from \eqref{eq:KNhol} that this is also true for all elements of $\fh$, thus $e_4$ is invariant under all elements of $\fh$. Therefore $\RR e_4$ is a non-degenerate subspace which is invariant under $\fh$, thus $\fh$ is decomposable. 

As for the second assertion, if $\Lambda_1=0$, then \eqref{c15} and $R_{ij}=[\Lambda_i,\Lambda_j]-\Lambda([e_i,e_j])$ imply $R_{15}=0$. Furthermore, $[\Lambda_4,\Lambda_5]=[h(0,v_4,0),\Lambda_5]=h(0,-\tr(A_5)v_4,0)$ by \eqref{eq:rhoh1} and thus $R_{45}=h(0,-\tr(A_5)v_4,0)-\sqrt v_4\Lambda_4$ by \eqref{c45} and $\Lambda_1=0$. That is, $\pro_y(R_{45})=0$. By Table \ref{tab:t1} this implies $R_{45}=0$. 
\end{proof}

\begin{lm}\label{lm:A1}
    If $\Lambda(\fg)\subset\fhIII$ and if $2(\tr(A_{r}))^2+\det(A_{r})\not=0$ and $6(\tr(A_{r}))^2+\det(A_{r})\not=0$  for $r=1$ or $r=4$, then $\fh$ is decomposable.
\end{lm}
\begin{proof}
    For $r\in \{1,4\}$, using Table~\ref{tab:t1} for $\nabla_{e_r} R$ we get  \begin{eqnarray*}
    \tr(\pro_A (B_{r56})) &=\ \tr(\pro_A ((-\nabla_{e_r} R)(e_5,e_6))) &=\ (-2a_{r1}-a_{r4})\tr(A_{56})-a_{r2}\tr(A_{57}),\\
    \tr(\pro_A(B_{r57}))& =\ \tr(\pro_A ((-\nabla_{e_r} R)(e_5,e_7)))&=\ (-a_{r1}-2a_{r4})\tr(A_{57})-a_{r3}\tr(A_{56}),
    \end{eqnarray*}
hence 
\begin{equation}\label{eq:1567}
(A_r+\tr(A_r)I_2)(\tr(A_{56}), \tr(A_{57}))=0.
\end{equation}
Moreover, $0=\pro_y (B_{r67})=\pro_y(-(\nabla_{e_r}R)(e_6,e_7))$ gives 
\begin{equation}\label{eq:y167}
(A_r+2\tr(A_r)I_2)(w_1,w_2)=v_r(\tr(A_{56}), \tr(A_{57})).
\end{equation}
Since $\det(A_r+n\tr(A_r)I_2)=n(n+1)(\tr(A_r))^2+\det(A_r)$ for $n\in\NN$, our assumption ensures that \eqref{eq:1567} and \eqref{eq:y167} imply $w_1=w_2=\tr(A_{56})=\tr(A_{57})=0$. The assertion now follows from Lemma~\ref{lm:Rdec}.
\end{proof}

Making use of the fixed basis, we define
\begin{equation}\label{eq:gidef}\fg_1:=\Span\{e_1\},\ \fg_2:=\Span\{e_2,e_3\},\ \fg_j:=\{e_1,\dots,e_j\},\ j=3,\dots,7.
\end{equation}

As a consequence of Proposition \ref{pr:brackets} and Table \ref{A1.h3} one can deduce:

\begin{pr}\label{pro:subalg}
If $\Lambda(\fg)$ is contained in $\fhIII$, then $\fg_j$ are subalgebras of $\fg$ for $j=1,\dots,5$ and $\fg_3$ is an ideal in $\fg_4$. Moreover, $[e_1,e_4]\in\RR e_1$. 
\end{pr}
\begin{re}\label{re:rep}
    Assume that $\Lambda(\fg)\subset\fhIII$. Table \ref{tab:t1} shows that $A_{ij}=0$ for $i,j=1,\dots,5$. Thus $e_j\mapsto A_j$, $j=1,\dots,5$, defines a representation of $\fg_5$.
\end{re}

\begin{re}\label{rem:vi} Assume that $\Lambda(\fg)\subset \fhIII$. Then $\Lambda_i=h(A_i,v_i,y_i)$ as in \eqref{eq:h3}. If $v_i=0$ for all $i=1, \ldots, 7$, then $\fh$ is decomposable. Indeed, if this holds then $\Lambda_i(e_4)=0$ for all $i=1, \ldots, 7$. This shows that the left-invariant vector field defined by $e_4$ is parallel. Hence $e_4$ is an invariant of the holonomy representation. Thus $\RR e_4$ is a non-degenerate subrepresentation of~$\fh$, so $\fh$ is decomposable.  
\end{re}

In the following we will sometimes change the basis of $\fg$ in order to simplify calculations. Then instead of $\mathfrak{e}=(e_1,\dots, e_7)$ we use a basis $ \mathfrak{e'}=(e_1,\dots, e_7)\cdot\Phi$, where $\Phi$ is a an element of $\G$ which preserves $\RR e_1$ and $\Span\{e_2,e_3\}$. These transformations constitute a subgroup of $\G$ whose Lie algebra is exactly $\fhIII$. In particular, $\Phi \fhIII \Phi^{-1}=\fhIII$. Let $\Lambda'$ denote the Levi-Civita connection with respect to $\mathfrak{e}'$ and $\fh'$ the holonomy algebra with respect to $\mathfrak{e}'$. Then $\Lambda'(X)=\Phi^{-1}\Lambda(X)\Phi\in\fhIII$ for all $X\in\fg$ and $\fh'=\Phi^{-1}\fh\Phi\subset\fhIII$. 

 In particular, we will use the following transformations:
\begin{eqnarray}
   \Phi_T &:=& \label{eq:tgl} \diag(\det T, T, 1, (\det T)^{-1}, (T^\top)^{-1}),\quad T\in\GL(2), \\
   \Phi_v&:=&\label{eq:tv} \exp h(0,v,0), \quad v\in \RR.
\end{eqnarray}
Transformations of the form $\Phi_v$ allow to add an arbitrary multiple of $e_1$ to $e_4$ and to extend this to a map in $\G$ which leaves $e_1,e_2,e_3$ invariant. 

For the next result we shall introduce some notation. Let $J$ denote the operator $$J(x,y,z):=[[x,y],z]+[[y,z],x]+[[z,x],y]$$ for $x,y,z\in\fg$. Furthermore, we put $J_{ijk}:=J(e_i,e_j,e_k)$ and define $J_{ijk}^l$ by $J_{ijk}=J_{ijk}^l e_l$.

\begin{lm}\label{lm:v1}
    If $\Lambda(\fg)\subset \fhIII$ and if $\tr(A_1)=0$ or $\tr(A_4)=0$, then $v_1=0$.
\end{lm}
\begin{proof}
    The assertion follows from the Jacobi identities $J_{145}^4=2 v_1^2-\sqrt2v_4\tr(A_1)=0$ and $J_{145}^5=(\tr (A_4)-\sqrt2 v_1)\tr(A_1)=0$.
\end{proof}

\begin{pr}\label{pr:Atrhdec}
 If $\Lambda(\fg)\subset \fhIII$ and $\fg_3$ defined in \eqref{eq:gidef} is not abelian, then $\fh$ is decomposable or the matrices $A_i$, $i=1, \ldots, 7$ are either all lower or all upper triangular.

\end{pr}

\begin{proof}
  By Proposition \ref{pr:brackets}  the Lie bracket relations are as in Table \ref{A1.h3}. Assume that $\fg_3$ is not abelian.

 By \eqref{c23}, $[e_2,e_3]=\lambda e_2+\mu e_3$, for some $\lambda,\mu\in\mathbb R$. If $[e_2,e_3]\neq 0$  then we can apply a transformation $\Phi_T$  as defined in \eqref{eq:tgl}, for $T=\begin{pmatrix}\lambda&\mu'\\\mu&\lambda'\end{pmatrix}$, where $\lambda'=\lambda/(\lambda^2+\mu^2),\mu'=-\mu/(\lambda^2+\mu^2)$, in order to obtain $[e_2,e_3]=e_2$. So we split the proof into two cases: $[e_2,e_3]=0$ or $[e_2,e_3]=e_2$. 

    
{\bf Case 1:} $[e_2,e_3]=0$. 

Depending on whether $\fg_2=\Span\{e_2,e_3\}$ is preserved by $e_1$ or not, we split into cases 1.1 and~1.2:

Case 1.1: $[e_1,\fg_2]\subset\fg_2$. By  \eqref{c12}, \eqref{c13} and \eqref{c23} we get the following identities
\begin{equation}
    \label{eq:case11}
    \tr(A_2)=\tr(A_3)=0, \; a_{24}=a_{33}, \; a_{22}=a_{31}.
\end{equation}
Then $\fg_3=\fg_1\ltimes \fg_2$, where $\fg_2$ is abelian and $e_1$ acts by an arbitrary linear map on the abelian ideal $\fg_2$. 
Since $\fg_3$ is not abelian,  the adjoint action of $e_1$ on $\fg_2$ is non-trivial.  It is given by $A_1$ with respect to the basis $e_2,e_3$ of $\fg_2$. If $\det(A_1)>0$ or if $\det(A_1)=0$ and $\tr(A_1)\not=0$, then $\fh$ is decomposable by Lemma~\ref{lm:A1}. So let us assume that $\det (A_1) < 0$ or that $\det(A_1)=\tr(A_1)=0$. Then $A_1 $ has two real eigenvalues that are either both equal to zero or both unequal to zero with opposite signs. We may change the basis of $\fg$ by a transformation $\Phi_T$ as in \eqref{eq:tgl}. Therefore we may assume that the action of $e_1$ has a certain normal form.  This leads to the following two cases, which correspond to the two possibilities for the eigenvalues of $A_1$.

Case 1.1.1: $\fg_3=\{[e_1,e_3]=e_2\}$ (The brackets between $\{\,\cdot\,\}$ are the only non-zero Lie bracket relations defining $\fg_3$.)

In addition to \eqref{eq:case11}, the form of $\fg_3$ introduces the additional identities
\begin{equation*}
a_{13}=a_{11}=a_{14}=0, \; a_{12}=1.
\end{equation*} In particular, this 
gives $v_1=0$  by Lemma~\ref{lm:v1}.

Proposition \ref{pro:subalg} implies that $\ad (e_4)$ is a derivation of $\fg_3$ and $[e_4,e_1]\in \RR e_1$, so we have $$\ad (e_4)=\left(\begin{array}{ccc} -r & 0 & -t \\ 0& -r-s& u \\ 0& 0 &-s \end{array}\right).$$ 
We may add a multiple of $e_1$ to $e_4$ (extendable to a map of the form $\Phi_v$ in \eqref{eq:tv}). Therefore we may assume $u=0$.
By \eqref{c14}
we get $\tr(A_4)=-r$. On the other hand $a_{41}=-r-s$ by \eqref{c24}, and $a_{44}=-s$ by \eqref{c34}. We obtain $s=0$. Consequently,
\[\fg_4=\{[e_1, e_3]=e_2, [e_1,e_4]=re_1, [e_2, e_4]=re_2, [e_3,e_4]=te_1\},\] which implies, by \eqref{c24} and \eqref{c34},
 \begin{eqnarray*}
  &a_{42}=a_{43}=a_{44}=0,\ v_2=0.&
\end{eqnarray*}
Furthermore, we already got $v_1=0$.
 We will show that $a_{j3}=0$ for all $j=1,\dots,7$. We already know that $a_{13}=a_{43}=0$.  Using Remark~\ref{re:rep} we obtain $[A_1,A_2]=0$ and $[A_1,A_5]=y_{12}A_3+y_{11}A_2-\tr(A_5)A_1$, which gives
 $a_{24}=a_{33}=a_{23}=0$ and $a_{53}=y_{12}a_{31}$. Note that $J_{137}^3=-a_{73}=0$. Now $J_{137}^2=2a_{74} + y_{12}$ together with $J_{157}^3=-y_{11}a_{73}+y_{12}(y_{12}-a_{74})=y_{12}(y_{12}-a_{74})$ implies $y_{12}=a_{74}=0$. Hence $a_{53}=0$ and finally, $J_{136}^3=a_{74}-a_{63}=-a_{63}$ gives $a_{63}=0$.
 Therefore, $A_i$ is upper triangular for all $i=1, \ldots, 7$.

Case 1.1.2: $\fg_3=\{[e_1,e_2]=e_2, [e_1, e_3]=r e_3\}$, $r<0$

In this case, in addition to \eqref{eq:case11}, we have  $A_1=\diag(1,r)$.
Using this and the identities $[A_1,A_2]=A_2$, $[A_1,A_3]=rA_3$ (see Remark~\ref{re:rep}), we get $A_2=A_3=0$. Moreover, $[A_1,A_4]=(\sqrt2 v_1-\tr(A_1))A_1$ yields $a_{42}=0$.

From the Jacobi identity we get
\begin{eqnarray*}
 &J_{136}^2=a_{62}(2-r)\quad  
J_{236}^2=-y_{31},\quad J_{135}^2=-2y_{31}+2a_{52},&
\nonumber
\end{eqnarray*}  which implies $a_{52}=a_{62}=0$. Thus $a_{j2}=0$ for all $j=1, \ldots,6$. Using \eqref{eq:R} we compute the curvature term $R_{173}^2=a_{72}-v_1a_{22}=a_{72}$ which must be zero by Table \ref{tab:t1}.
Therefore all $A_i$ are lower triangular. \

Case 1.2: $[e_1,\fg_2]\not\subset\fg_2$ 

Let $[e_1,e_2]=re_1+se_2+te_3$ and $[e_1,e_3]=pe_1+qe_2+le_3$.
We may assume $r=1$, $p=0$. $J_{123}=0$ implies 
$[e_1,e_3]=0$, which leads to $\fg_3=\{[e_1,e_2]=e_1+se_2+te_3\}.$

Here we have $a_{11}=s, a_{13}=t$, $a_{12}=a_{14}=0$. If $s\not=0$, then $\fh$ is decomposable by Lemma~\ref{lm:A1}. So let us assume that $s=0$. Then the
 structure of $\fg_3$ implies 
\[a_{11}=a_{12}=a_{14}=0,\ a_{13}=t,\ \tr(A_2)=-1,\ \tr(A_3)=0,\ a_{24}=a_{33},\ a_{22}=a_{31}. \]
In particular, $v_1=0$ by Lemma~\ref{lm:v1}. We will first show that $t=0$. Using Remark \ref{re:rep}, we have that $[A_1,A_2] = A_1 + t A_3$ and $[A_2, A_3]=0$. This implies that $t a_{31}=0$, $t(3a_{24}+2)=0$ and 
 \begin{equation}\label{eq:A23-rep}
2a_{23}a_{31}+2a_{24}^2+a_{24}=0.\end{equation} 
 Thus $t=0$, otherwise we obtain a contradiction.
 
 Furthermore, since $t=0$, we have $J_{124}^1=a_{41}=0$, $J_{134}^1=a_{42}=0$, and $J_{234}^1=\sqrt2 v_3=0$. Since we can add a multiple of $e_1$ to $e_4$, we may assume that $[e_2,e_4]$ does not have an $e_1$-component, thus $v_2=0$. Now we find 
\[J_{245}^1=-2\sqrt2v_5,\ J_{245}^4=\sqrt2 v_4,\]
which gives $v_1=\ldots=v_5=0$.  If $a_{44}\not=0$, then $\det(A_4)=0$ and $\tr(A_4)\not=0$, hence $\fh$ is decomposable by Lemma~\ref{lm:A1}.  
 So let $a_{44}=0$ in the following. This implies $a_{43}=0$. Indeed, from $[A_2, A_4]=-a_{43}A_3$, we conclude that $a_{31}a_{43}=0$ and $(3a_{24}+1)a_{43}=0$. If $a_{43} \neq 0$, this would contradict \eqref{eq:A23-rep}. Now  $J_{245}=y_{41}e_2+y_{42}e_3$ yields $y_{4j}=0$,  $j=1,2$.  Now $J_{256}^4+J_{347}^1=2\sqrt2 v_6$ implies $v_6=0$. Assume $v_7\not=0$, then 
$$J_{246}=\sqrt2 a_{31}v_7e_1=0,\ J_{247}=\sqrt2 v_7(a_{24}-1)e_1=0$$ 
would give $a_{31}=0$ and $a_{24}=1$, which again contradicts  \eqref{eq:A23-rep}. Therefore, $v_7=0$ and from the above all $v_i$ vanish, so $\fh$ is decomposable by Remark \ref{rem:vi}.

{\bf Case 2: }$[e_2,e_3]=e_2$.

We have $J_{123}^1=\tr(A_2)=0$. This implies  $a_{13}=-J_{123}^3-J_{125}^5=0$, i.e. $[e_1,e_2]=r e_2$,  which gives $[e_1,e_3]\in\Span\{e_1,e_2\}$ by Jacobi identity.  Thus  $a_{14}=0$. In particular $\tr(A_1)=a_{11}$ and $\det (A_1)=0$. If $a_{11}\not=0$,  Lemma~\ref{lm:A1} implies that $\fh$ is decomposable. So let us assume that $a_{11}=0$. From the bracket relations of $\fg_3$ we conclude that 
\begin{eqnarray*}
A_1 = \begin{pmatrix}
    0 & q\\ 0 &0
\end{pmatrix},\  a_{21}=-a_{33},\ a_{22}=1+a_{31},\ a_{24}=a_{33}.
\end{eqnarray*}

The above identities imply $J_{234}^3=- a_{43}=0$ and  $v_1=0$ by Lemma~\ref{lm:v1}.

 Case 2.1.: $q\neq 0$

Inspecting the Jacobi identities, and using $q\neq 0$, we obtain $a_{44}=0=v_2$ from $J_{134}^2$ and from $J_{135}^4$, respectively. The term $J_{126}$ yields $a_{73}=a_{23}=a_{33}=0$. Combining $J_{135}^3$ and $J_{125}^2$, we get $y_{12}=0$. From $J_{137}^2$ with $J_{156}^5$ results $a_{71}=a_{74}=0$. Finally, from $J_{156}^6$ and $J_{126}^2$, we conclude $a_{53}=a_{63}=0$. 
Therefore, all $A_i$ are upper-triangular.

Case 2.2.: $q=0$ 

In this case $A_1=0$, so 
$J_{234}^2=a_{44}=0$. 
This implies $J_{125}^2=y_{12}=0$ and $J_{135}^2=-y_{11}=0$, thus $\Lambda_1=0$. In particular, we obtain $R_{15}=0$. 
If $a_{41}\neq 0$ then $\fh$ is decomposable by Lemma \ref{lm:A1}. So we suppose $a_{41}=0$ for the remaining of the case.

Let us first assume that $a_{42}=0$, i.e. $A_4=0$.  From $J_{245}^2$ we obtain that $y_{42}=0$. Thus, if also $y_{41}=0$, we would have $\Lambda_1=0$ and $\Lambda_4=h(0,v_4,0)$, implying $\fh$  decomposable by Lemma \ref{lm:Rdec}. So let us assume that $y_{41}\neq 0$. From Remark \ref{re:rep}, we have that $[A_4,A_5]=y_{41}A_2$. Then $y_{41}\neq 0$ implies $a_{23}=a_{33}=0$ and $a_{31}=-1$. Since $J_{345}^2=(2-a_{34})y_{41}$ then $a_{34}=2$. From $[A_3,A_5]=-A_5+(y_{32}-a_{54})A_3$, we deduce $a_{53}=0$. Now, from $J_{234}^1$ we get $v_2=0$ and $J_{12k}^1$ yields $a_{k3}=0$ with $k=6,7$.  Thus all $A_i$, $i=1,\ldots, 7$, are upper triangular.

Let us now suppose that $a_{42}\neq 0$. Since $[A_3, A_4]=-a_{42}A_2$, then $a_{23}=a_{33}=0$,  $a_{34}=2a_{31}+1$ and $[A_2, A_3]=A_2$ implies $a_{31}(a_{31}+1)=0$. From $[A_3,A_5]$ we obtain  $a_{53}(2a_{31}+1)=0$ which implies $a_{53}=0$. Taking $J_{234}^1=v_2(3a_{31}+2)$ and $J_{246}^3=-a_{42}a_{73}$, we conclude $v_2=a_{73}=0$. If $a_{31}\neq 0$, then $J_{267}^7=a_{63}a_{31}$ implies $a_{63}=0$. If $a_{31}=0$, then we have 
\begin{equation*}
    J_{246}^2=-(a_{63}+a_{71})a_{42},\quad  J_{346}^3=(a_{63}-a_{74})a_{42}, \quad J_{237}^2=a_{74}-a_{71}
\end{equation*}
from which we finally obtain $a_{63}=0$. Thus all $A_i$, $i=1\cdots 7$, are upper triangular.
\end{proof}

\section{Exclusion of the case that ${\mathfrak a}$ is diagonal}

Let $G$ be a 7-dimensional Lie group endowed with a parallel left-invariant $\G$-structure $\omega$. Assume that its holonomy algebra $\fh$ is indecomposable of type III. By Theorem  \ref{TFK},  $\fh$ is conjugate to $\fa\ltimes \fm(1,0,k)$  where $\fa$ can be $0,\,\fsl(2,\RR),\, \fgl(2,\RR),\, \fco(2),\fd$, or $\RR\cdot\diag(1,0)$, and $\fm(1,0,k)$ is as in \eqref{eq:m} for some $k\in \{1,2\}$. 
This section is devoted to showing that the last two cases for $\fa$ cannot occur in our left-invariant context.

Let us be more precise. Let $\fh_e\subset\fso(\fg,\ip)$ be an indecomposable holonomy algebra of type~III. Then there exists a basis $\mathfrak{e}=(e_1,\dots,e_7)$ such that the holonomy with respect to $\mathfrak{e}$,  which we denote $\fh$, is equal to one of the Lie algebras $\fa\ltimes\fm(1,0,k)\subset \fhIII$ listed in Theorem~\ref{TFK} and such that $\omega=\omega_0$, where $\omega_0$ is as in \eqref{Eomega}. In the following, $\fa$ will always denote this Lie algebra $\fa\subset\fgl(2,\RR)$ uniquely determined by $\fh_e$. 

In the basis $\mathfrak e$ above, the Levi-Civita connection verifies $\Lambda(X)\in \fhI$  for all $X\in \fg$ by Proposition \ref{pro:educg}. Thus for each $i=1, \ldots, 7$, we set $\Lambda_i=h(A_i, v_i, (u_{i1},u_{i2}),(y_{i1},y_{i2}))$ as in \eqref{eq:lami}. In addition, the curvature $R$ is an element of $\mathcal K(\fhIII)$.  
From Table~\ref{tab:t1}, we know that $\fa\subset\fd$ implies 
\begin{equation}\label{eq:Rijzero}
0=R_{1i}=R_{2j}=R_{3k}=R_{4l}
,\quad
\forall i \neq 5,\; \forall j \neq 6,\; \forall k \neq 7, \; \forall l\neq 5.
\end{equation}
In addition, using the notation in \eqref{eq:havuy},
\begin{eqnarray}
    R_{15}&=&h(0, 0,(r,s)),\label{eq:R15d}\\
    R_{26}&=&h(0,0,(r,0)),\\
    R_{37}&=&h(0,0,(0,s)),\label{eq:R37d}\\
    R_{45}&=& h(0, 0,(\sqrt2 w_1,\sqrt2 w_2)),\label{eq:R45d_d}\\
    R_{56}&=&h\left(\diag(r,0),w_1,(j_3,t)\right),\label{eq:R56}\\
    R_{57}&=&h\left(\diag(0,s), w_2,(t,j_4)\right),\label{eq:R57}\\
    R_{67}&=&h(0, 0,(-w_1,-w_2)). \label{eq:R67d}
\end{eqnarray}

\begin{pr}\label{pro:diagg}
Let  $\fh_e\subset\fso(\fg,\ip)$ be indecomposable of type {\rm III}. Then $\fa \neq \fd$. 
\end{pr}
\begin{proof}
  Let $\mathfrak{e}$ and $\fh$ be as above and assume $\fa=\fd$. Proposition~\ref{pro:educg} implies that $A_i\in\fd$ for $i=1,\dots,7$,  i.e. $A_i=\diag(a_{i1},a_{i4})$. In particular,  by \eqref{eq:R56} and \eqref{eq:R57}, $\pro_A([\Lambda_i, R_{56}])=\pro_A([\Lambda_i, R_{57}])=0$ for all $i=1, \ldots, 7$.
Thus, $r\neq 0$ and $s\neq 0$ in \eqref{eq:R15d}--\eqref{eq:R67d}, since otherwise we would get $\fa\subsetneq\fd$.  

Computing the curvature tensor using \eqref{eq:R} and considering $R_{26}$ and $R_{37}$ as above, respectively, we have that 
\begin{eqnarray*}
    r = (a_{61}+a_{64}-y_{11})y_{21} - v_2y_{31} - (3a_{21}+2a_{24})y_{61}\\
    s = (a_{71}+a_{74}-y_{12})y_{32} + v_3y_{22} - (3a_{31}+2a_{34})y_{72}
\end{eqnarray*}
We will make use of the second Bianchi identity \eqref{eq:2B} to infer that $r=s=0$, thus leading to a contradiction. 

From $B_{125}=0$ and from $B_{135}=0$ we get that $A_2=A_3=0$. Now, considering $0=\pro_y(B_{257})$  we conclude $y_{21}=y_{22}=0$. Also, from $\pro_y(B_{356})$ we conclude $y_{31}=y_{32}=0$. Thus $r=s=0$, and the contradiction follows. 
\end{proof}

\begin{pr}\label{pr:anotdiag}
    Let   $\fh_e$ be indecomposable and of type {\rm III}. Then $\fa\not=\RR\cdot\diag(1,0)$.
\end{pr}
\begin{proof}  Let $\mathfrak{e}$ and $\fh$ be as above and assume $\fa=\RR\cdot\diag(1,0)$.  Using that $[x,y]=\Lambda(x)y-\Lambda(y)x$ one can show that the Lie bracket relations between the basis elements $e_1, \ldots, e_7$ in $\mathfrak e$ are as in Table~\ref{A1.rd}. 
Note that Equations \eqref{eq:Rijzero}  --
\eqref{eq:R67d}
hold for $s=0$.
By Proposition \ref{pro:educg}, for $i=1,\dots,7$, the map $\Lambda_i$ has the form \[ \Lambda_i=h(\diag(a_{i1},a_{i4}),v_i,(0,u_{i2}),(y_{i1},y_{i2}))\in\fhI. \] This, together with \eqref{eq:KNhol}, implies that $r$ appearing in \eqref{eq:Rijzero}--\eqref{eq:R67d} verifies $r\neq 0$. Indeed, if $r=0$, then $\fr_0=\Span \{R(x,y)\mid x,y\in \fg\}\subset \fm(1,0,2)$ and thus 
$\fh\subset \fm(1,0,2)$, which implies $\fa=0$ contradicting our hypothesis. Now, we proceed as follows:

Step 1: $a_{k4}=-2a_{k1}$ holds for $k=1,2,3,4,7$: 

We project onto $\fa$ the second Bianchi identity in \eqref{eq:2B} for $(i,j,k)=(5,6,k)$, $k=1,\dots,4,7$. Recall that $A_{ij}=\pro_A(R_{ij})$, $i,j=1,\dots,7$. Since $A_{56}$ is the only $A_{ij}$ that does not vanish, this easily gives 
$0=(-2a_{k1}-a_{k4})A_{56}$, which implies the assertion.

Step 2: We have $y_{11}=(2a_{61}+a_{64})/4$ and  $y_{21}=-(2a_{51}+a_{54})/4$, and $y_{31}=0$:

Using \eqref{eq:R15d} -- \eqref{eq:R67d} we compute 
the first component of $\pro_y$ applied to both sides of \eqref{eq:2B}. This gives $0=4ry_{11}-(2a_{61}+a_{64})r$  for $(i,j,k)=(1,5,6)$, $0=4ry_{21}+(2a_{51}+a_{54})r$ for $(i,j,k)=(2,5,6)$, and $0=y_{31} r$ for $(i,j,k)=(3,5,6)$. Since $r\neq 0$, we obtain our claim.

Step 3: $v_i=-a_{i1} w_1/r$ for $i=1,\dots,4$:

Equations  \eqref{eq:Rijzero}, \eqref{eq:R37d} and Table \ref{tab:t1} for $\nabla_{e_6} R$ immediately imply  $(\nabla_{ e_6}R)_{7i}=0$ for $
i=1,2,3,4$. Furthermore, $(\nabla_{ e_7}R)_{i6}=0$ holds for $i=1,3,4$. For $i=2$, the first component of $\pro_y$ of the latter term equals $(2a_{71}+a_{74})r$, which vanishes by Step 1. Finally, the first component of $\pro_y(R(\Lambda_i e_6,e_7)+R( e_6,\Lambda_i e_7)-[\Lambda_i,R_{67}])$ is equal to $-a_{i1}w_1-v_ir$ for $i=1,\dots,4$. Now \eqref{eq:2B} gives the assertion.

Step 4: $A_3=0$, $u_{32}=0$: 

We have $0=R_{13}=[\Lambda_1,\Lambda_3]-\Lambda(a_{14}e_3-\tr(A_3)e_1)=[\Lambda_1,\Lambda_3]+\Lambda(2a_{11}e_3-a_{31}e_1)$ by Step 1. This implies
\[
   0=R_{132}^2=3a_{11}a_{31},\quad
   0=R_{132}^1=-a_{31}u_{12},\quad
   0=R_{135}^2=-a_{31}y_{11}, 
\]
where we used $y_{31}=0$ for the third identity. On the other hand, we have
$$R_{15}=[\Lambda_1,\Lambda_5]-\Lambda(u_{12}e_6-\tr(A_1)e_5+\sqrt2 v_1e_4+y_{12}e_3+y_{11}e_2-\tr(A_5)e_1).$$
Using Steps 1 and 3 and $y_{31}=0$ from Step 2 we obtain,  taking suitable $\alpha, \beta, \gamma$,
\begin{equation}\label{nn}R_{155}^2=\alpha a_{11}+\beta y_{11}+\gamma u_{12}=r\not=0.
\end{equation}
This implies $a_{31}=0$, thus $A_3=0$. In particular, 
$0=R_{23}=[\Lambda_2,\Lambda_3]-\Lambda(a_{24}e_3+u_{32}e_1).$
Proceeding as above we obtain 
$u_{32}a_{11}=u_{32}u_{12}=u_{32}y_{11}=0$, thus \eqref{nn} gives $u_{32}=0$.

Step 5: $A_1=0$:

First we consider $R_{16}=[\Lambda_1,\Lambda_6]+\Lambda(a_{11}e_6+v_1e_3+(y_{11}+\tr(A_6))e_1)$, which implies
$A_{16}=a_{11}A_6+(y_{11}+\tr(A_6))A_1$ since $A_3=0$.  This gives 
\[a_{11}(2a_{61}+a_{64}+y_{11})=5a_{11}y_{11}=0. 
\]
Assume that $a_{11}\not=0$. Then the latter equation implies $y_{11}=0$.
Using Steps 1 and 2, we now get 
\[
R_{165}^{2}=a_{11}y_{61}-a_{61}y_{11}+y_{11}^2=a_{11}y_{61}=0,
\]
thus $y_{61}=0$.
This together with Steps 1 and 2 gives
\[
R_{265}^2=r=a_{61}y_{21}\neq 0,\quad R_{125}^2=0=-a_{11}y_{21},
\] which is a contradiction. Hence $a_{11}=0$, thus $A_1=0$.

Step 6: $A_4=0$:

Since $A_1=A_3=0$, we have $A_{24}=a_{41}A_2$, and $A_{46}=a_{41}A_6$. This gives
\begin{equation}\label{nnn} 0=R_{242}^2=a_{41}a_{21},\quad 0=R_{462}^2=a_{41}a_{61},\quad 0=R_{463}^3=a_{41}a_{64}.\end{equation}
On the other hand, $R_{26}=[\Lambda_2,\Lambda_6]+\Lambda(a_{21}e_6+v_6e_3+a_{61}e_2+(y_{21}-u_{62})e_1)$. Since $y_{13}=0$ and $y_{11}=(2a_{61}+a_{64})/4$, we obtain, for suitable $\alpha$, $\beta$, $\gamma$,
\[0\not=r=R_{265}^2=\alpha a_{21}+\beta a_{61}+\gamma a_{64}.\]
Together with \eqref{nnn} this gives $a_{41}=0$, thus $A_4=0$.

Step 7: $a_{64}=2a_{61}\not=0$:

We already know that $a_{11}=0$ and therefore also $v_1=0$ by Step 3. This implies $R_{16}=[\Lambda_1,\Lambda_6]+(y_{11}+\tr(A_6))\Lambda_1$, hence
\[0=R_{165}^2=-(a_{61}+\tr(A_6))y_{11}+(y_{11}+\tr(A_6))y_{11}=\frac14(-2a_{61}+a_{64})y_{11}.\]
Thus in order to prove $a_{64}=2a_{61}$ it suffices to prove $y_{11}\not=0$. Assume $y_{11}=0.$ Then 
\[R_{15}=[\Lambda_1,\Lambda_5]-\Lambda(u_{12}e_6+y_{12}e_3-\tr(A_5)e_1).\]
In particular, $A_{15}=-u_{12}A_6$ because of $A_1=A_3=0$. This gives $u_{12} a_{61}=u_{12}a_{64}=0$. Moreover, 
$0\not=r=R_{155}^2=-u_{12}y_{61}$, thus $u_{12}\not=0$. Consequently, $a_{61}=a_{64}=0$.
Thus we have $A_6=0$ besides $A_1=A_3=A_4=0$, which gives
$\diag(r,0)=A_{56}= y_{61}A_2$. However, this contradicts $a_{24}=-2a_{21}$ since $r\not=0$. Furthermore, the latter consideration also shows that $A_6\not=0$,  which is equivalent to $a_{61}\neq 0$.

Step 8: $A_2=0$, $u_{12}=0$: 

Since $A_1=A_3=A_4=0$, we have
$0=A_{26}=a_{21}A_6+a_{61}A_2$, thus $a_{21}a_{61}=0$. Since $a_{61}\not=0$ this implies $a_{21}=0$ and therefore also $A_2=0$. Furthermore, we have $0=-R_{162}^1=-a_{64}u_{12}+(y_{11}+\tr(A_6))u_{12}=2a_{61}u_{12}$, which implies $u_{12}=0$. 

Step 9: final contradiction:

Since $A_1=A_2=A_3=A_4=0$ we obtain  
$\diag(r,0)=A_{56}=(a_{51}+u_{62})A_6-\tr(A_6)A_5$. Together with $a_{64}=2a_{61}$ this gives 
\begin{eqnarray*}
r&=& (-2a_{51}+u_{62})a_{61},\\
0&=& (2a_{51}-3a_{54}+2 u_{62})a_{61},
\end{eqnarray*}
hence
\begin{equation}\label{cont}
    2r=3(-2a_{51}+a_{54})a_{61}.
\end{equation}
On the other hand, $R_{15}=[\Lambda_1,\Lambda_5]-\Lambda(y_{12}e_3+y_{11}e_2-\tr(A_5)e_1)$ since $u_{12}=0$ and $A_1=0$, which also implies $v_1=0$ by Step 3. Since $A_1=0$ by Step 5 and $y_{31}=0$ by Step 2, we get
$r=R_{155}^2=-(\tr(A_5)+a_{51})y_{11}-y_{11}y_{21}+\tr(A_5)y_{11}=-y_{11}(y_{21}+a_{51}).$
Applying Step 2 and Step 7, we obtain
\[ 4r={a_{61}}(-2a_{51}+a_{54}),\]
which contradicts \eqref{cont} since $r\not=0$.
\end{proof}

\begin{co} \label{co:triang} 
Let $\fh_e$ be indecomposable and of type {\rm III}. If all $A_i$ are upper triangular or all $A_i$ are lower triangular, then $\fa=0$.
\end{co}
\begin{proof}   Let $\mathfrak{e}$,  $\fh$ and $\fa$ be as above. By \eqref{eq:m0} and \eqref{eq:KNhol}, $\fa$ is contained in the subalgebra of $\fgl(2,\RR)$ spanned by $A_i$, $i=1,\dots,7$.  If all $A_i$ are upper triangular or all are lower triangular, then  we have $\fh=\fa\ltimes \fm(1,0,k)$ with $\fa$ being a subalgebra of the triangular $2\times 2$-matrices. However, by Theorem \ref{TFK}, the only such subalgebras are 
$0,\,\fd,$ or \linebreak $\RR\cdot\diag(1,0)$. Since the last two cases are ruled out in Propositions \ref{pro:diagg} and \ref{pr:anotdiag}, we get $\fa=0$.
\end{proof}

\section{ Proving that ${\mathfrak a}$ is trivial}

 Let again $\fh_e\subset\fso(\fg,\ip)$ be an indecomposable holonomy algebra of a left-invariant metric and assume that $\fh_e$ is of type~{\rm III}.  As in the previous section, let $\mathfrak{e}=(e_1,\dots,e_7)$ be a basis of $\fg$ such that $\omega=\omega_0$, where $\omega_0$ is as in \eqref{Eomega} and the holonomy with respect to this basis is one of the Lie algebras $\fh=\fa\ltimes\fm(1,0,k)\subset \fhIII$ listed in Theorem~\ref{TFK}. 
\begin{lm}\label{lm:anzimpliesg3ab}
 Let $\fh_e$ be indecomposable and of type {\rm III}. If $\fa\not=0$, then 
 \begin{itemize}
     \item[(i)] $\Lambda(\fg)\subset\fhIII$, and therefore $\Lambda_i=\Lambda(e_i)=h(A_i,v_i,(y_{i1},y_{i2}))$ for some
     $$A_i=\begin{pmatrix} a_{i1} & a_{i2} \\ a_{i3}& a_{i4}\end{pmatrix}, \quad v_i\in \RR, \quad (y_{i1},y_{i2})\in \RR^2, \quad i=1, \ldots, 7.$$ 
     \vspace{-0.5cm}
     \item[(ii)] the Lie subalgebra $\fg_3$ defined in \eqref{eq:gidef} is abelian.
 \end{itemize}
\end{lm}
\begin{proof}
 If $\fa$ is non-trivial, then by Proposition~\ref{pr:anotdiag}, $\fa$ is one of the Lie algebras in item~1 of Theorem~\ref{TFK}. In this case $\Lambda(\fg)\subset\fhIII$ by Proposition~\ref{pro:educg}. This gives (i). In particular, we can apply Proposition~\ref{pr:Atrhdec}, which leads to the following conclusion. Assume that $\fg_3$ were not abelian. Then all $A_i$ are upper triangular or all $A_i$ are lower triangular. Now Corollary \ref{co:triang} implies $\fa=0$, which contradicts our assumption.
\end{proof}

\begin{lm}\label{lm:anzimpliesA4z}
Let $\fh_e$ be indecomposable and of type {\rm III}. 
If $\fa\not=0$, then $A_4=0$.
\end{lm}
\begin{proof}
Assume that $\fa\neq 0$.
Then Lemma \ref{lm:anzimpliesg3ab} implies that $\Lambda(\fg)\subset \fhIII$ and that $\fg_3$ is abelian.
Together with Eqns.~\eqref{c12}, \eqref{c13} and \eqref{c23} this gives
$$A_1 = 0,\ a_{21}=-a_{33},\ a_{22}=a_{31},\ a_{24}=a_{33},\ \tr(A_3) = 0.$$
Since $\tr(A_1)=0$, from  Lemma~\ref{lm:v1}, we obtain $v_1=0$. 

For a contradiction, suppose  that $A_4\neq 0$. We may apply Lemma~\ref{lm:A1} for $r=4$: If $\det(A_4)>0$ or if $\det(A_4)=0$ and $\tr(A_4)\not=0$, then $\fh$ would be decomposable. Thus $\det (A_4) < 0$ or $\det (A_4)=\tr(A_4)=0$. Then $A_4$ has two real eigenvalues that are either both equal to zero or both unequal to zero with opposite signs. By choosing appropriate basis vectors $e_2$ and $e_3$, we may assume that $A_4$ is given in normal form. 

Case 1: $A_4 = \begin{pmatrix}
    r & 0\\
    0 & s 
\end{pmatrix}$,  $rs<0$

Analyzing the Jacobi equations, we get that $J_{235}=0$ implies $v_2=v_3=0$. From $J_{246}^6=-ra_{33}$ and $J_{246}^3=-v_4a_{33}-a_{63}$ we obtain $a_{33}=a_{63}=0$. Similarly, $J_{247}^6=(2r-s)a_{23}$ and $J_{247}^3=v_4a_{23}+(r-2s)a_{73}$ imply $a_{23}=a_{73}=0$. Furthermore, $J_{245}^3=2s(y_{22}-a_{53})$ yields $a_{53}=y_{22}$. Using $a_{23}=a_{33}=0$ we obtain $J_{247}^1=-2s y_{22}$, which finally gives $a_{53}=0$. Thus  all $A_j$ are upper triangular, which contradicts Corollary~\ref{co:triang}.

 Case 2: $A_4 =\begin{pmatrix}
    0 & 1\\
    0 & 0 
\end{pmatrix}$

Here, we obtain  from $J_{235}$ that $v_2=0$. That $a_{23}=a_{33}=a_{73}=0$ results from $J_{246}^3=J_{246}^6=J_{246}^7=0$. From $J_{236}^7$ we obtain $a_{31}=0$ and from $J_{457}^7$ we conclude $a_{53}=0$.  Finally, $-2a_{63}=2J_{467}^7+J_{456}^5-J_{347}^2=0$. This shows  that all $A_i$ are upper triangular matrices, which contradicts Corollary~\ref{co:triang}.
\end{proof}

\begin{pr} \label{pro:anzimpliesred} 
If $\fh_e$ is indecomposable and of type {\rm III}, then $\fa =0$.
\end{pr}
\begin{proof}
Assume for a contradiction that $\fa\neq 0$. Then $\Lambda(\fg) \subset \fhIII$, $\fg_3$ is abelian and $A_4=0$ by Lemmas \ref{lm:anzimpliesg3ab} and \ref{lm:anzimpliesA4z}.

Since $\fg_3$ is abelian, \eqref{c12}, \eqref{c13} and \eqref{c23} imply  $A_1=0$, $\tr(A_2)=\tr(A_3)=0$ and $a_{21}=-a_{33}$, $a_{22}=-a_{34}$. In addition,  $v_1=0$  by Lemma~\ref{lm:v1}. 

By Remark~\ref{re:rep} and the fact that $\fg_3$ is abelian, we get $[A_2,A_3]=0$.
Therefore, $\Span\{A_2,A_3\}$ is an abelian subalgebra of  $\fsl(2,\RR)$, thus of dimension~$\le 1$.

By Remark~\ref{rem:Aspa}, the subalgebra $\fa$ of the holonomy algebra $\fh$ is contained in the Lie algebra spanned by $A_i$, $i=2,3,5,6,7$ (since $A_1=A_4=0$). We shall prove that these matrices span a Lie algebra of dimension $\leq 1$, from which we obtain $\fa=0$ by Theorem \ref{TFK} and Proposition~\ref{pr:anotdiag}. We shall split the proof in cases.

{\bf Case 1: }$A_2\not=0$ or $A_3\not=0$. Then there exists a non-trivial linear combination $\lambda A_2+\mu A_3=0$.  After a change of basis in $\fg$ by a transformation $\Phi_T$ for $$T:=\begin{pmatrix}
\mu&\lambda\\
-\lambda &\mu
\end{pmatrix}
$$ as defined in \eqref{eq:tgl} we have $A_2\neq 0$ and $A_3=0$. Since $a_{21}=-a_{33}$, $a_{22}=-a_{34}$ and $\tr(A_2)=0$, as shown above, we obtain 
$$A_2=\begin{pmatrix}
0&0\\
a_{23} &0
\end{pmatrix}
$$ with $a_{23}\neq 0$.

Recall the notation $A_{ij}=\pro_A(R_{ij})$, for $i,j=1, \ldots, 7$.  From Table \ref{tab:t1}, \eqref{eq:Aij} and Eqns.~\eqref{c15}, \eqref{c45} we get 
 $$A_{15}=-y_{11}A_2=0,\quad A_{45}=-y_{41}A_2=0,$$ thus $y_{11}=y_{41}=0$.
By Table \ref{tab:t1} and \eqref{c47}, we obtain
$A_{47} =- v_4A_2=0,
$ so $v_4=0$. Moreover, by Table \ref{tab:t1} and \eqref{c27},
 $$
[A_2,A_7]+a_{23} A_6-(v_2-a_{71}) A_2=A_{27} \in \fsl(2,\RR),
$$ which
implies $\tr(A_6)=0$. By 
$A_{25} = [A_2,A_5]+(a_{51}-y_{21})A_2=0$
we get that $[A_5,A_2]$ is a multiple of  $A_2$, so strictly lower triangular. This implies $a_{52}=0$, since $a_{23}\neq 0$.
Using the table again one gets 
 \begin{equation}\label{eq:R2637}
A_{37}=(a_{72}-v_3) A_2=  -A_{26}
 = -[A_2,A_6]-a_{61}A_2.
\end{equation} 
Hence $[A_2,A_6]$ is a multiple of $A_2$ and therefore strictly lower diagonal. This implies $a_{23}a_{62}=0$, thus $a_{62}=0$ and $[A_2,A_6]=2a_{61}A_2$. Equation~\eqref{eq:R2637} now gives $(a_{72}-v_3)A_2=-3a_{61}A_2$, so $v_3-a_{72}=3a_{61}$.
However, $J_{237}^3=0$ gives $v_3=a_{61}$ so we conclude $a_{72}=-2a_{61}$. Using this and Table \ref{tab:t1}, we get 
$$
A_{67}=[A_6,A_7]+a_{61}A_7+(a_{63}-a_{71})A_6-v_6A_2
=0,
$$
As $A_6$, $A_2$ are lower triangular, the upper right corner of the matrix in the last equation gives $a_{61}=0$. Therefore, $A_2,A_5,A_6,A_7$ are all lower triangular. By Corollary~\ref{co:triang}, the holonomy $\fh$  would be decomposable,  which is a contradiction.

{\bf Case 2:} $A_2=A_3=0$.

 We may assume that $y_{11}\in\{0,1\}$ and $y_{12}=0$. Indeed, if one of the numbers is different from zero, we may assume that $y_{11}$ is this number by possibly interchanging $e_2$, $e_3$ and $e_6$, $e_7$ using $\Phi_T$ for $T=\begin{pmatrix} 0&1\\1&0\end{pmatrix}$ (see \eqref{eq:tgl}). After that, a base change by $\Phi_T$  for $T=\begin{pmatrix}
    y_{11}&0\\y_{12}&\frac1{y_{11}}
\end{pmatrix}$  yields $y_{11}=1$. 

We now consider the cases $y_{11}=1$ and $y_{11}=0$ separately. Recall that $\fa$ is generated by $A_5,A_6,A_7$, since $A_i=0$ for $i=1, \ldots, 4$.
\begin{enumerate}
\item $y_{11}=1$. By $J_{245}^2=\sqrt2 v_2$ and $J_{345}^2=\sqrt2 v_3$, we get $v_2=v_3=0$. 
From $J_{15i}^j=0$ for $i=6,7$, $j=2,3$ we get
$a_{61}=1$, $a_{63}=a_{71}=a_{73}=0.$ In particular, $A_6,A_7$ are both upper triangular. If $a_{53}=0$, then $A_i$ is upper triangular for all $i=1,\dots,7$, thus $\fh$ is decomposable or $\fa=0$ by Corollary~\ref{co:triang}. So we assume $a_{53}\neq 0$.
If $a_{74}\not=0$, then 
$J_{167}^1=(a_{72}-a_{64})a_{74}$, $ J_{257}^3=2(a_{53}-y_{22})a_{74}$ would imply $a_{64}=a_{72}$ and $y_{22}=a_{53}$, which would lead to the contradiction  $0=J_{157}^1-J_{567}^6=-2 a_{53}\not=0$. Hence $a_{74}=0$. Now 
$$J_{567}^6=(a_{72}-2a_{64}-1)a_{53},\quad J_{367}^2=a_{72}^2-a_{72}=0$$
implies that either $a_{72}=0$, $a_{64}=-1/2$
or $a_{72}=1$, $a_{64}=0$. Both cases lead to a contradiction. Indeed, in the first case we obtain $0=J_{357}^3=a_{53}/2\not=0$ and in the second one $0=J_{157}^1+J_{357}^3=-3a_{53}\not=0$.

\item $y_{11}=y_{12}=0$ so that $\Lambda_1=0$ and thus $R_{15}=0$. 
By Table~\ref{tab:t1}, $\tr(A_{56})=\tr(A_{57})=0$ and thus $\fa\subset \mathfrak{sl}(2,\RR)$. By Theorem \ref{TFK}, $\fa$ is either $\mathfrak{sl}(2,\RR)$  or $\fa=0$.

\begin{enumerate}
\item $v_4=0$. If $y_{41}=y_{42}=0$, then $\Lambda_4=0$ and  Lemma \ref{lm:Rdec}, implies that $\fh$ is decomposable.  So either $y_{41}$ or $y_{42}$ is non-zero. Similarly to the discussion for $(y_{11},y_{12})$ we see that  we can assume $y_{41}=1$ and $y_{42}=0$ from now on. One has $\pro_v (R_{45})=-v_2$. Thus $v_2=0$ by Table~\ref{tab:t1}. From $J_{456}^3=J_{457}^3=0$, we obtain $a_{63}=a_{73}=0$, thus $A_6$, $A_7$ are upper triangular. If $a_{53}$ vanishes, then also $A_5$ is upper triangular and we are done by Corollary~\ref{co:triang}. So we assume $a_{53}\neq 0$. 
The identies $J_{456}^2=J_{457}^2=0$ yield $a_{64}=-2a_{61}$ and $a_{74}=-2a_{71}$. 

If $a_{71}=0$, then also $a_{74}=0$. Hence $J_{157}^1=-a_{53}(a_{61}+a_{64})$, 
which implies $a_{61}+a_{64}=0$. Thus $a_{61}=a_{64}=0$. Now $J_{567}^6=0$ yields $a_{72}=0$.  Moreover, $J_{346}^1=0$ gives $v_3=0$, which implies $J_{367}^1+ J_{256}^2=a_{53}a_{62}$, hence $a_{62}=0$.  We obtain $A_6=A_7=0$. Consequently, $\fa\subset \RR\cdot A_5$, which gives $\fa=0$.

If $a_{71}\not=0$, then $J_{347}^1=-\sqrt2 v_3 a_{71}$ implies $v_3=0$. Moreover, $J_{367}^3=-2a_{71}(a_{72}+3a_{61})$ gives $a_{72}=-3a_{61}$. Using the condition $J_{567}^6=0$ we now obtain $\tr(A_5)=0$. From $J_{156}^1=J_{157}^1=0$ we get
\begin{equation}\label{eq:32a}
a_{61} a_{51}+a_{52}a_{71}=0,\quad a_{53}a_{61}-a_{71}a_{51}=0.  
\end{equation}
After an explicit computation, this implies $[A_5,A_7]=3a_{71}A_5$.
Multiplying $J_{367}^2=4(3a_{61}^2+a_{62}a_{71})$ by $a_{53}/a_{71}$ and using \eqref{eq:32a} we obtain 
\begin{equation}\label{eq:32a1}
a_{53}a_{62}+3a_{51}a_{ 61}=0.
\end{equation}
 Eq. \eqref{eq:32a} together with \eqref{eq:32a1} imply that the non-trivial linear combination $a_{53}A_6 - a_{51}A_7$ vanishes.
So in this case  $A_5,A_6,A_7$ span a 2-dimensional Lie subalgebra of $\mathfrak{gl}(2,\RR)$ containing $\fa$. As mentioned above, the only possibilities for $\fa$ are $\fsl
(2,\RR)$ or $\fa=0$. Therefore $\fa=0$.

\item $v_4\neq 0$. Here $J_{467}^2=J_{467}^3=0$ gives $v_2=-\tr(A_7)$ and $v_3=\tr( A_6)$. Then $J_{156}^1=J_{157}^1=J_{256}^4=J_{257}^4=J_{346}^1=J_{347}^1=0$ shows that $(-\tr(A_7),\tr(A_6))$ is in the common kernel of $A_5^\top$, $A_6^\top$ and $A_7^\top$. If this vector is non-zero, then there exists a basis such that $A_5,A_6,A_7$ are all lower triangular and the proposition follows from  Corollary~\ref{co:triang}. 

Let us assume now that $\tr(A_7)=\tr(A_6)=0$. Then $v_2=v_3=0$. Using this we conclude from $\pro_v(R_{45})=0$ that $\tr(A_5)=-\sqrt 2 v_4$. 

From $J_{567}^4=J_{567}^7=J_{567}^6=0$ we have
\begin{eqnarray}
-(a_{63}-a_{71})v_6+(a_{61}+a_{72})v_7&=&0, \label{eq:a}\\
(a_{63}-a_{71})a_{52}+(a_{61}+a_{72})a_{51}&=&0,\\
(a_{63}-a_{71})a_{54}+(a_{61}+a_{72})a_{53}&=&0. \label{eq:b}
\end{eqnarray}
Recall that $R_{15}=0$. Since $\fh$ is indecomposable, we get $w_1\not=0$ or $w_2\not=0$ by Lemma~\ref{lm:Rdec}. On the other hand, $w_1=\pro_v (R_{56})=a_{51}v_6+ a_{52}v_7$ and $w_2=\pro_v (R_{57})=a_{53}v_6+ a_{54}v_7$. Thus  the system of linear equations \eqref{eq:a} -- \eqref{eq:b} for $a_{63}-a_{71}$ and $a_{61}+a_{72}$ has only the trivial solution. Hence
\begin{equation} a_{63}=a_{71}, \quad a_{61}=-a_{72}.\label{eq:bii}\end{equation}
 This implies $0=A_{67}=[A_6,A_7]-\pro_A(\Lambda([e_6,e_7]))=[A_6,A_7]$. Since $A_6,A_7\in\fsl(2,\RR)$, we see that $A_6$ and $A_7$ are linearly dependent. Applying a suitable transformation $\Phi_T$, we may assume $A_7=0$. Equation~\eqref{eq:bii} implies 
\[A_6=\begin{pmatrix} 0&a_{62}\\ 0&0\end{pmatrix}.\]
Moreover, $A_7=0$ implies $\pro_{y}R_{37}=0$ and therefore $b_4=0$. 
We compute $A_{57}=a_{53}A_6$, which gives $0=b_4=a_{62}a_{53}$. If $a_{62}=0$, then the Lie algebra generated by $A_5,A_6,A_7$ is at most one-dimensional, thus $\fa=0$. If $a_{53}=0$, then all $A_j$, $j=1,\dots,7$, are upper triangular, hence $\fa=0$ also in this case. 
\end{enumerate}
\end{enumerate}
\end{proof}

\section{Examples and proof of Theorem~\ref{th:main}}
\label{S7}
In this section we first present examples of Lie groups endowed with parallel $\G$-structures whose holonomies are the abelian Lie algebras appearing in Theorem \ref{th:main}. Later we conclude with the detailed proof of this main theorem, using the results in the previous sections.

Let us consider the one-parameter family of  Lie algebras  $\fg_\eps$ defined by 
\begin{eqnarray}
  &&\textstyle  [e_1, e_5] = \frac12 e_1,\ [e_2, e_5] = -\frac12 e_2,\ [e_2,e_6]=-\frac12 e_1,\ [e_3, e_5] = 2 e_3, \nonumber\\
  &&\textstyle  [e_3, e_7] = -\frac12 e_1,\ [e_4, e_5] = - e_2,\  [e_4, e_6] = -e_1, \label{eq:geps}\\
   && \textstyle [e_5, e_6] = -\eps e_2-2 e_4-e_6,\ [e_5, e_7] = -\eps e_3+\frac32 e_7,\ [e_6, e_7]= -\sqrt2 e_2,\nonumber
\end{eqnarray}
for $\varepsilon \in \mathbb R$. 

Obviously,  $\fp:=\Span\{e_1,\dots, e_4\}$ is an abelian ideal of  $\fg_\eps$. Thus  $\fg_\eps$ is an extension of  $\fs:= \fg_\eps/\fp$ by  $\fp$. The Lie algebra  $\fs$ is isomorphic to $\{ [X,Y]=-Y, [X,Z]=\frac32 Z\}$. The Lie algebra $\fg_\eps$ can also be understood as a semi-direct product of the ideal $\fn$ spanned by $e_j$, $j\not=5$, by $\RR$. Let $G_\eps$ be the simply connected Lie group with Lie algebra  $\fg_\eps$ and $N$ the nilpotent subgroup with Lie algebra $\fn$. Then  $G_\eps={\RR} \ltimes N$. We write elements of $G_\eps$ as $(t,n)$ according to this decomposition.  Doing so, the left-invariant vector field $e_5$ is equal to $\partial_t$.  

 Let $\fg_\eps$ denote the Lie algebra defined by \eqref{eq:geps} endowed with the inner product \eqref{Eip} induced by the $\G$-structure $\omega_0$ in \eqref{Eomega}.  Given two parameters $\varepsilon$ and $\varepsilon'$, the Lie algebras $\fg_\eps$ and $\fg_{\eps'}$ are isomorphic. For instance, an isomorphism is given by $\varphi: \fg_\eps \longrightarrow \fg_{\eps'}$ such that $\varphi(e_i)=e_i$, for $i=1, \ldots, 5$, and $\varphi(e_6)=e_6+\tfrac{2}{3}(\eps'-\eps)e_2$,  $\varphi(e_7)=e_7+\tfrac{2}{7}(\eps-\eps')e_3$.
 We will prove that $\fg_\eps$ is isometrically isomorphic to $\fg_{\eps'}$ if and only if $\eps=\eps'$ by showing that the parameter $\eps$ is determined by the Lie algebra structure and the inner product. In the following we will choose vectors $f_4,f_5,f_6,f_7$ only using the Lie algebra structure and the inner product. After choosing, each of these vectors will be described with respect to the basis used in~\eqref{eq:geps}. The center $\fz$ of $\fg_\eps$ is one-dimensional and non-degenerate. We choose $f_4\in\fz$ such that $\langle f_4,f_4\rangle=-1$. Then $f_4=\pm(e_4-2e_2)$. There exists exactly one further one-dimensional subspace $\mathfrak{u}$ which is invariant under the adjoint representation, namely $\mathfrak{u}=\RR\cdot e_1$.  We choose an element $f_5$ which acts by multiplication with $-1/2$ on $\mathfrak{u}$. Then $f_5=e_5+f_5'$, where $f_5'\in\Span\{e_1,\dots,e_4,e_6,e_7\}$. The action of $f_5$ on $\mathfrak{q}:=[\fg_\eps,\fg_\eps]/\fg_\eps^{(2)}=\Span\{e_1,e_2,e_3,2e_4+e_6,e_7\}/\Span\{e_1,e_2\}$ is independent of  the choice of $f_5$. It is diagonalizable with eigenvalues $-2$, $-1$ and $3/2$. We denote the corresponding eigenspaces by $E_{-2}$, $E_{-1}$ and $E_{3/2}$. Choose a representative of an eigenvector in $E_{-1}$ such that $|\langle f_6,f_4\rangle|=4$. Then $f_6=\pm(e_6+2e_4)+a_6e_1+b_6e_2$ for some $a_6,b_6\in\RR$. Note that the inner product on $\fg_\eps$ induces a non-degenerate inner product on $\mathfrak{q}$, and let $f_7$ be a representative of a non-zero vector in an isotropic complement of $E_{-2}$ in $E_{-1}^\perp\subset\mathfrak{q}$.  We `normalize' $f_7$ by requiring that $\langle [f_6,f_7],f_6\rangle=-\sqrt 2$, which yields $f_7=e_7+a_7 e_1+ b_7 e_2$ for some $a_7,b_7\in \RR$. Then $\langle [f_5,f_7],f_7\rangle$ depends only on $\lb$ and $\ip$, but not on the choices we made. Our assertion now follows from 
 $\langle [f_5,f_7],f_7\rangle=-\eps.$

Consider the left-invariant metric on $G_\eps$ defined by \eqref{Eip}. The linear maps $\Lambda_i$ belong to $\fhIII$. We have $\Lambda_1=0$ and
\begin{eqnarray}
&& \Lambda_2=h(0,0,y_2),\quad \ y_2=(1/2,0),   \label{eq:l1} \\
&& \Lambda_3=h(0,0,y_3),\quad \ y_3=(0,1/2),   \\
&& \Lambda_4=h(0,0,y_4),\quad \ y_4=(-1,0),     \\
&& \Lambda_5=h(A_5,0,0),\quad A_5=\diag( 1,-3/2),   \\
&& \Lambda_6=h(0,v_6,y_6),\quad  v_6=\sqrt2,\ y_6=(\eps,0),\\
&& \Lambda_7=h(0,0,y_7),\quad \ y_7=(0,\eps).   \label{eq:l5}
\end{eqnarray}
The holonomy of our metric is equal to $\fm(1,0,2)$ for $\eps\not=0$ and to $\fm(1,0,1)$ for $\eps=0$. 
The vectors $e_1, e_2, e_3\in  T_e G_\eps$ are invariants of the holonomy representation. The parallel vector fields corresponding to these vectors are not left-invariant. In terms of the left-invariant vector fields $e_1,e_2,e_3$, they are equal to
 $e^{\frac12t} e_1$, $e^{-t} e_2$, and $e^{\frac32 t} e_3$, respectively. Indeed, e.g.
 $$\textstyle\nabla_{e_5} (e^{\frac12t} e_1)=\partial_t(e^{\frac12t}) e_1+e^{\frac12t}\Lambda_5(e_1)=\frac12 e^{\frac12t} e_1-e^{\frac12t} \cdot\frac12 e_1=0.$$ 
 
We want to deform $\omega_0$ by a pointwise orthonormal transformation $A$ depending on $t$. In the end we will obtain  a family of parallel $\G$-structures that are not left-invariant but induce the same left-invariant metric like $\omega_0$. Consider $M:=M(a,b,c)\in\fso(4,3)$ given by $e_1,e_2,e_3\in \ker M$ and 
\begin{eqnarray*}
    &\textstyle M(e_4)=a \exp(\frac12 t) e_1+b\exp(-t) e_2+c\exp(\frac32 t) e_3,&\\ &\textstyle M(e_5)=a \exp(\frac12 t) e_4,\quad M(e_6)=b\exp(-t) e_4,\quad M(e_7)=c\exp(\frac32 t) e_4&
\end{eqnarray*}
and put $A:=A(a,b,c):=\exp M(a,b,c)$. Now we obtain a 3-parameter family of three-forms
\begin{equation}\label{eq:3f}
\omega(a,b,c)=A(a,b,c)^*\omega_0.
\end{equation}
Since $A$ is orthogonal, all these 3-forms induce the same metric as $\omega_0$. Moreover, they are parallel. This can be seen as follows. Instead of showing that $A^*\omega_0$ is parallel we prove that $(A^{-1})^*\nabla(A^*\omega_0)$ vanishes. Since $\omega_0$ is parallel, this is equivalent to 
\begin{equation}\label{eq:Ao}
\omega_0(D_le_i,e_j,e_k)+\omega_0(e_i,D_le_j,e_k)+\omega_0(e_i,e_j,D_le_k)=0
\end{equation}
for all $i,j,k,l\in\{1,\dots,7\}$, where $D_l=A\circ \nabla_{e_l}A^{-1}$. 

Equations \eqref{eq:l1} -- \eqref{eq:l5} give $D_l=0$ for $l\not=6$, $D_6e_i=0$ for $i=1,2,3,4,$ and
\begin{eqnarray*}
    D_6(e_5)&=&\textstyle  2b\exp(-t) e_2+2c\exp(\frac32 t)e_3,\\ \quad D_6(e_6)&=&-2b \exp(-t) e_1,\\ D_6(e_7)&=&\textstyle -2c \exp(\frac32 t) e_1, 
\end{eqnarray*}
which immediately implies \eqref{eq:Ao}.

Although we have neither explained nor worked with the relation between spinors and $\G$-structures in this paper let us now roughly explain the idea behind the above construction using it. For details on this relation see \cite{Kg2}. Every (not necessarily parallel) $\G$-structure on a Lie group $G$ that induces the same metric on $G$ as the $\G$-structure $\omega$ is given by a 3-form $A^*\omega$, where $A$ is function from $G$ to $\grO(\fg,\ip)\cong\grO(4,3)$. If $A$ maps to $\SO(4,3)$, then $A^*\omega$ induces the same orientation as $\omega$, otherwise we get the opposite orientation.  Let us first consider the situation at the identity $e\in G$. The $\SO(4,3)$-orbit of $\omega$ in $\bigwedge^3(\fg^*)$ corresponds bijectively to the set of non-isotropic elements in the projective real spinor representation $(\Delta_{4,3}\setminus \{0\})/\RR^*=(\RR^{4,4}\setminus \{0\})/\RR^*$ of ${\rm Spin}(4,3)$. Thus it can be identified with $(S^{4,3}/\ZZ_2)\cup (H^{3,4}/\ZZ_2)$, where $S^{4,3}$ and $H^{3,4}$ are the sets of space- and time-like unit vectors in $\RR^{4,4}$, respectively. Furthermore, it is well known that given a non-isotropic spinor $\psi_0\in\Delta_{4,3}$ the Clifford multiplication $X\mapsto X\cdot\psi_0$ defines an isometry from $\fg$ to $\psi_0^\perp\subset\Delta_{4,3}$. 

Now we again turn to 3-forms and spinor fields on $G$. Then of course a $\G$-structure $A^*\omega$ is parallel if and only if the associated unit spinor field (unique up to sign) is parallel. Consider $\omega$ as a parallel 3-form on $G$ and let $\psi_0$ be the unit spinor field that corresponds to $\omega$. Assume that there exists a parallel isotropic vector field $X$ on $G$. Then $\psi':=\psi_0+X\cdot \psi_0$ is a parallel unit spinor field. Hence it corresponds to a parallel $\G$-structure $\omega'=A^*\omega$ for some function $A:\fg\to\SO(4,3)$. In the present example we applied this construction to the 3-dimensional space of isotropic parallel vector fields, which gave us a 3-parameter family of parallel $\G$-structures. 

Conversely, let $\omega'=A^*\omega$ be parallel and let $\psi'$ be the corresponding parallel unit spinor field (unique up to sign). We decompose $\psi'=f\cdot\psi_0 +\psi^\perp$, where $f$ is a function on $G$ and $\psi^\perp\perp\psi_0$. Since $\psi'$ and $\psi_0$ are parallel, $f$ is constant and $\psi^\perp$ is also parallel. Thus we can define a parallel vector field $X$ by $\psi^\perp=X\cdot\psi_0$. If the holonomy $\fh$ is indecomposable, then $X$ is isotropic. In particular, the three-parameter family in~\eqref{eq:3f} contains all parallel $\G$-structures that induce the same metric and the same orientation on  $G_\eps$ as $\omega_0$. The parallel $\G$-structures that induce the same metric and but the reverse orientation on $G$ are then exactly the 3-forms $-\omega(a,b,c)$ for $a,b,c\in\RR$. 
  
\begin{proof}[Proof of Theorem~\ref{th:main}]
Let $\fh$ be a subalgebra of $\g$ whose natural representation on $\RR^{4,3}$ is indecomposable and of type \III. By Lemma~\ref{lm:abelian}, $\fh$ is abelian if and only if it is conjugate to the two-dimensional Lie algebra $\fm(1,0,1)$ or the three-dimensional Lie algebra $\fm(1,0,2)$. In particular, this proves the second assertion of the Theorem. Furthermore, the holonomy algebras of the examples described above are $\fm(1,0,1)$ and $\fm(1,0,2)$. This shows one direction of the first assertion.

Assume now that $\fh$ is an indecomposable type III infinitesimal holonomy  algebra of a left-invariant $\G$-structure. Then, on the one hand, $\fh$ is on the list in Theorem \ref{TFK}. On the other hand, $\fa=0$ holds by Proposition~\ref{pro:anzimpliesred}. Thus  $\fh=\fm(1,0,1)$ or $\fh=\fm(1,0,2)$.
\end{proof}

\appendix
\section{Appendix}\label{A1}
\subsection{Algebraic curvature tensors with values in ${\fhIII}$}\label{Acurv}
\begin{small}
\begin{table}[H]
\begin{center}
\renewcommand{\arraystretch}{1.4}
\begin{tabular}{|c|c|c|c|}
\hline
$R(e_i,e_j)$&$A$&$v$&$y$\\
\hline
$R_{15}$&$0$&$0$&$(b_1+b_4,b_3+c_4)$\\[1ex]
$R_{26}$&$\left(\begin{array}{cc} -a_1&-a_2 \\-a_3&a_1 \end{array}\right)$&$0$&$(b_1,b_3)$\\[3.5ex]
$R_{27}$&$\left(\begin{array}{cc} -a_3& a_1\\j_1&a_3 \end{array}\right)$&$0$&$(b_3,c_3)$\\[3.5ex]
$R_{36}$&$\left(\begin{array}{cc} -a_2& j_2\\a_1&a_2 \end{array}\right)$&$0$&$(b_2,b_4)$\\[3.5ex]
$-R_{67}=\frac1{\sqrt2}R_{45}$&$0$&$0$&$(w_1,w_2)$\\[1ex]
$R_{56}$&$\left(\begin{array}{cc} b_1&b_2 \\b_3& b_4\end{array}\right)$&$w_1$&$(j_3,t)$\\[3.5ex]
$R_{57}$&$\left(\begin{array}{cc} b_3&b_4\\c_3&c_4 \end{array}\right)$&$w_2$&$(t,j_4)$\\[3.5ex]
\hline
\multicolumn{4}{|c|}{$R_{12}=R_{13}=R_{14}=R_{16}=R_{17}=R_{23}=R_{24}=0$} \\[0.5ex]
\multicolumn{4}{|c|}{$R_{25}=R_{34}=R_{35}=R_{46}=R_{47}=0$}\\[0.5ex]
\multicolumn{4}{|c|}{$R_{37}=R_{15}-R_{26}$}\\[0.5ex]
\hline
\end{tabular}
\caption{Curvature}
\label{tab:t1}
\end{center}
\end{table}
\end{small}

\subsection{Commutators for $\Lambda_i\in\fhIII$, $i=1, \ldots, 7$.}\label{A1.h3}

\begin{eqnarray}
\,[e_1, e_2] &=&  a_{13}e_3+a_{11}e_2-\tr (A_2) e_1 \label{c12}\\
\,[e_1, e_3] &=& a_{14}e_3+a_{12}e_2-\tr (A_3) e_1 \label{c13}\\
\,[e_1, e_4] &=&  (\sqrt{2} v_1-\tr (A_4))e_1 \label{c14}\\
\,[e_1, e_5] &=&  -\tr (A_1) e_5+\sqrt{2}v_1e_4+y_{12}e_3+y_{11}e_2-\tr (A_5) e_1, \label{c15}\\
\,[e_1, e_6] &=&  -a_{12}e_7-a_{11}e_6-v_1 e_3+(-y_{11}-\tr (A_6))e_1, \label{c16}\\
\,[e_1, e_7] &=& -a_{14} e_7-a_{13} e_6+v_1 e_2+(-y_{12}-\tr (A_7)) e_1, \label{c17} \\
\,[e_2, e_3] &=& (a_{24}-a_{33}) e_3+(a_{22}-a_{31} )e_2,\label{c23}\\
\,[e_2, e_4] &=&   -a_{43} e_3-a_{41} e_2+\sqrt{2} v_2 e_1, \label{c24}
\end{eqnarray}
\begin{eqnarray}
\,[e_2, e_5] &=& -\tr (A_2)e_5+\sqrt{2}v_2e_4+(y_{22}-a_{53})e_3+(y_{21}-a_{51})e_2,\label{c25}\\
\,[e_2, e_6] &=&  -a_{22}e_7-a_{21}e_6+(-v_2-a_{63})e_3-y_{21}e_1-a_{61}e_2, \label{c26}\\
\,[e_2, e_7] &=& -a_{24} e_7-a_{23} e_6-a_{73} e_3+(v_2-a_{71}) e_2-y_{22} e_1, \label{c27}\\
\,[e_3, e_4] &=&  -a_{44} e_3-a_{42} e_2+\sqrt{2} v_3 e_1, \label{c34}\\
\,[e_3, e_5] &=&  -\tr (A_3)e_5+\sqrt{2}v_3e_4+(y_{32}-a_{54})e_3+(y_{31}-a_{52})e_2,\label{c35}\\
\,[e_3, e_6] &=&  -a_{32} e_7-a_{31} e_6+(-v_3-a_{64}) e_3-y_{31} e_1-a_{62} e_2, \label{c36}\\
\,[e_3, e_7] &=&  -a_{34} e_7-a_{33} e_6-a_{74} e_3+(v_3-a_{72}) e_2-y_{32} e_1, \label{c37}\\
\,[e_4, e_5] &=&  -\tr (A_4) e_5+\sqrt{2} v_4 e_4+y_{42} e_3+y_{41} e_2-\sqrt{2} v_5 e_1, \label{c45}\\
\,[e_4, e_6] &=& -a_{42} e_7-a_{41} e_6-v_4 e_3+(-y_{41}-\sqrt{2} v_6) e_1,\\
\,[e_4,e_7]&=&  -a_{44} e_7-a_{43}e_6+v_4 e_2+(-y_{42}-\sqrt{2} v_7) e_1,\label{c47}\\
\,[e_5, e_6] &=&  -a_{52}e_7-a_{51}e_6+\tr(A_6)e_5-\sqrt{2} v_6 e_4-(v_5+y_{62}) e_3-y_{51} e_1-y_{61}e_2,\qquad\ \ \label{c56}\\
\,[e_5, e_7] &=&  -a_{54} e_7-a_{53} e_6+\tr(A_7) e_5-\sqrt{2} v_7 e_4-y_{72} e_3+(v_5-y_{71}) e_2-y_{52} e_1,\label{c57}\\
\,[e_6, e_7] &=&  (-a_{64}+a_{72})e_7+(-a_{63}+a_{71})e_6+v_7e_3+v_6e_2+(-y_{62}+y_{71})e_1 \label{c67}
\end{eqnarray}

\subsection{Commutators for $\fa=\RR\cdot\diag(1,0)$}\label{A1.rd}

\begin{eqnarray*}
\,[e_1, e_2] &=& a_{11}e_2-(u_{12}+\tr (A_2))e_1, \\
\,[e_1, e_3] &=& a_{14}e_3-\tr (A_3)e_1, \\
\,[e_1, e_4] &=& \sqrt2 u_{12}e_3+(\sqrt{2}v_1-\tr (A_4))e_1,\\
\,[e_1, e_5] &=& u_{12}e_6-\tr (A_1)e_5+\sqrt{2}v_1e_4+y_{12}e_3+y_{11}e_2-\tr (A_5)e_1,\\
\,[e_1, e_6] &=& -a_{11}e_6-v_1e_3+(-y_{11}-\tr (A_6))e_1,\\
\,[e_1, e_7] &=& -a_{14}e_7+\sqrt{2}u_{12}e_4+v_1e_2+(-y_{12}-\tr (A_7))e_1,\\
\,[e_2, e_3] &=& a_{24}e_3-a_{31}e_2+u_{32}e_1, \\
\,[e_2, e_4] &=& \sqrt{2}u_{22}e_3-a_{41}e_2+(\sqrt{2}v_2+u_{42})e_1,\\
\,[e_2, e_5] &=&u_{22}e_6-\tr (A_2)e_5+\sqrt{2}v_2e_4+y_{22}e_3+(y_{21}-a_{51})e_2+u_{52}e_1,\\
\,[e_2, e_6]& =& -a_{21}e_6-v_2e_3-a_{61}e_2+(-y_{21}+u_{62})e_1, \\
\,[e_2, e_7] &=& -a_{24}e_7+\sqrt{2}u_{22}e_4+(v_2-a_{71})e_2
+(-y_{22}+u_{72})e_1,\\
\,[e_3, e_4] &=& (\sqrt{2}u_{32}-a_{44})e_3+\sqrt{2}v_3e_1, \\
\,[e_3, e_5] &=& u_{32}e_6-\tr (A_3)e_5+\sqrt{2}v_3e_4+(y_{32}-a_{54})e_3+y_{31}e_2, \\
\,[e_3, e_6] &=& -a_{31}e_6-(v_3+a_{64})e_3-y_{31}e_1, \\
\,[e_3, e_7] &=& -a_{34}e_7+\sqrt{2}u_{32}e_4-a_{74}e_3+v_3e_2-y_{32}e_1,\\
\,[e_4, e_5]& =& u_{42}e_6-\tr (A_4)e_5+\sqrt{2}v_4e_4+(y_{42}-\sqrt{2}u_{52})e_3+y_{41}e_2-\sqrt{2}v_5e_1, \\
\,[e_4, e_6] &=& -a_{41}e_6-(v_4+\sqrt{2}u_{62})e_3-(y_{41}+\sqrt{2}v_6)e_1,\\
\,[e_4, e_7] &=& -a_{44}e_7+\sqrt{2}u_{42}e_4-\sqrt{2}u_{72}e_3+v_4e_2-(y_{42}+\sqrt{2}v_7)e_1,\\
\,[e_5, e_6] &=& -(a_{51}+u_{62})e_6+\tr (A_6)e_5-\sqrt{2}v_6e_4-(v_5+y_{62})e_3-y_{51}e_1-y_{61}e_2,\label{eq:56TI}\\
\,[e_5, e_7] &=& -a_{54}e_7-u_{72}e_6+\tr (A_7)e_5+\sqrt{2}(u_{52}-v_7)e_4-y_{72}e_3+
(v_5-y_{71})e_2-y_{52}e_1,\\
\,[e_6, e_7] &=& -a_{64}e_7+a_{71}e_6+\sqrt{2}u_{62}e_4+v_7e_3+v_6e_2+(-y_{62}+y_{71})e_1
\end{eqnarray*}

\begin{description}

\item[Viviana del Barco,] Instituto de Matem\'atica, Estat\'istica e Computa\c{c}\~ao Cient\'ifica, Universidade Estadual de Campinas, UNICAMP, Rua Sergio Buarque de Holanda, 651, Cidade Universitaria Zeferino Vaz, 13083-859, Campinas, S\~ao Paulo, Brazil.\\
{\it Email address:} {\tt delbarc@ime.unicamp.br }

 \item[Ana Cristina Ferreira,] Centro de Matemática, Universidade do Minho, Campus de Gualtar, 4710-057 Braga, Portugal. \\
 {\it Email address:} {\tt anaferreira@math.uminho.pt} 

\item[Ines Kath,] Institut f\"ur Mathematik und Informatik, Universit\"at Greifswald, Walther-Rathenau-Str. 47, D-17487 Greifswald, Germany \\
{\it Email address:} {\tt ines.kath@uni-greifswald.de}

\end{description}

\end{document}